\documentclass[11pt]{article}

\usepackage{amsmath,amssymb,amsthm}
\usepackage{color}
\usepackage{graphicx}

\newtheorem{conjecture}{Conjecture}
\newtheorem{definition}{Definition}

\newtheorem{proposition}{Proposition}
\newtheorem{remark}{Remark}

\def \beq{ \begin{equation}}
\def \eeq{\end{equation}}

\newcommand{\e}{\epsilon}
\title{Three-body relative equilibria on $\mathbb{S}^2$\\
I: Euler configurations}
\date{}  
\begin{document}
	\maketitle
	\author{\begin{center}
	{ Toshiaki~Fujiwara$^1$, Ernesto P\'{e}rez-Chavela$^2$}\\	
		\bigskip
	   $^1$College of Liberal Arts and Sciences, Kitasato University,       Japan. fujiwara@kitasato-u.ac.jp\\
	    $^2$Department of Mathematics, ITAM, M\'exico. ernesto.perez@itam.mx\\
	\end{center}
	

\bigskip

\begin{flushright}
{\it To the memory of our friend Florin Diacu,\\
who was the inspiration of this work}
\end{flushright}

\begin{abstract}
Using the properties of the angular momentum, we develop a new geometrical technique to study relative equilibria for a system of $3$--bodies with positive masses, moving on the two sphere under the influence of an attractive potential depending only on the mutual distances among the bodies. With the above techniques we do an
analysis of the relative equilibria for the case of three bodies when they are moving on the same geodesic (Euler configurations). 
\end{abstract}

{\bf Keywords} Relative equilibria, Euler configurations, cotangent potential, condition for the shape.

\section{Introduction}
We are interested in the study of the simplest solutions of three point positive masses, moving on $\mathbb{S}^2$ under the influence of an attractive potential, which only depends on the mutual distances among the particles. 
These kind of solutions are called {\it relative equilibria}, from here on we denote them as $RE$. For the classical Newtonian $3$--body problem on the plane, it is well known that there are two classes of $RE$,
Eulerian and Lagrangian, the names are given by honouring to Leonard Euler and Joseph Louis Lagrange. Euler discovered the first class in 1767 \cite{Euler}, corresponding to collinear solutions. 
Few years later Lagrange, rediscovered them and found a new class of 
$RE$, which are obtained when we put three arbitrary masses at the vertices of an equilateral triangle, and particular choice of initial velocities are given \cite{Moeckel}. The initial position of the masses is called a {\it central configuration}, the respective solution for the three body problem ($RE$) is getting when we take a particular uniform rotation through the center of mass. In the rotating frame they become an equilibrium for the equations of motion, from here the name \cite{Wintner}. In this paper we study these kind of motions when the particles move on $\mathbb{S}^2$. Of course, in this case collinear motions mean that the three masses are on the same geodesic.

In the last years, Eulerian and Lagrangian $RE$ have been studied into the framework of the so called {\it the curved $3$--body problem}, that is, when the three masses are moving on a space of constant curvature. In dimension two, for positive curvature they are contained in or analysis. Usually, the corresponding authors consider the masses moving under the influence of the cotangent potential, which is a natural extension of the Newtonian potential to $\mathbb{S}^2$ (see \cite{Diacu-EPC1} and the references therein for more details). This is the reason why, we exemplify our new geometrical technique taking the cotangent potential as the main example in this article. In this case we can compare the results that we obtain (applying our technique), with the  results obtained by other authors using different approaches. 

More concretely, along the paper we will assume that three positive masses are moving on a two sphere of radius $R$ under the influence of the potential
\begin{equation}\label{genpot}
V = \sum_{i<j} m_im_j U(D^2_{ij}),
\end{equation}
where $D_{ij}$ is the chord distance 
between the particles with masses $m_i$ and $m_j$ respectively.
We exemplify our results on the two sphere of radius $R$ using the cotangent potential given by
\begin{equation}\label{partpot}
U =(1/R) \cot (\sigma),
\end{equation}
where $\sigma$
is equal to the angle between the two particles 
as seen from the origin of $\mathbb{S}^2$,
see for instance 
\cite{JA,Borisov1,Borisov2,Borisov3,Carinena,Diacu-EPC1,Diacu-EPC2,Diacu1,Diacu2,Diacu3,Diacu4,Kozlov,EPC2,tibboel} among others.
We observe that the potential \eqref{partpot} is included in the wide class of potentials \eqref{genpot}. 

There are many articles about relative equilibria on the sphere for the positive curved problem,  
most of them for the case of the Kepler and the $2$--body problem. 
We point out here that, contrary to what happens in the planar case, these problems are not equivalent, the first one in integrable whereas the second one is not
\cite{Shchepetilov2,Vozmischeva2}, the reason for this fact is that the center of mass is not more a first integral on curved spaces.
Day by day the literature on the subject is increasing, you can find papers analyzing Eulerian and Lagrangian motions on $\mathbb{S}^2$, and even polygonal $RE$ on $\mathbb{S}^2$, see for instance \cite{Diacu1,Borisov1,M-S,EPC1,tibboel,zhu} and the references therein, but just for the cotangent potential. Until we know, this is the first time that the problem is tackled in a general setting. 

Since the material getting by the analysis of the $RE$ on $\mathbb{S}^2$ is extensive, we have decided to split it into two natural parts. In the first one, we just analyze the collinear $RE$ of three bodies, called by short the Eulerian $RE$, in a forthcoming paper we will include the analysis for the $RE$ out of a geodesic or the Lagrangian $RE$. 

We must emphasize that for the Newtonian $n$--body problem, fixing the center of mass at the origin, we get that the linear momentum is a first integral.
For motions on the sphere we lost this first integral,
but using the conditions $c_x = c_y = 0$ getting from the angular momentum for $RE$ on $\mathbb{S}^2$ (see Section \ref{sec3}), we  determine the $z$--axis for the rotation axis. This is the essence of the new geometrical technique that we develop in this work, for the analysis of
the rigid motions on $\mathbb{S}^2$, when three particles of arbitrary masses are moving under the influence of a wide class of potentials given by \eqref{genpot}.
The introduction of this technique by first time, 
is undoubtedly the main contribution of this article. 

After the introduction, the paper is organized as follows: In Section \ref{prelims}
we get the equations of motion for three particles moving on $\mathbb{S}^2$ in spherical coordinates, we also obtain the equations for the angular momentum, which will play a main role for our geometric analysis of the $RE$.

 In Section \ref{sec3}, we give the definitions of relative equilibria and rigid rotator. In some sense a rigid rotator plays the same role as the central configurations do for the analysis of the Newtonian 
 $RE$ on the plane.  We also show that for the $RE$, two components of the angular momentum vanish, that is $c_x=0$ and $c_y=0$, this fact is of the utmost importance
in our approach, since these components give us two new integrals of motion that we will exploit along the manuscript. 

In Section \ref{sec4}, 
in order to have concrete results, we describe the relative equilibria on the equator for the cotangent potential \eqref{partpot}.
The equator is the special place where both the shape of the triangle formed by the three masses and the configuration can be described by the differences of the longitudes between the masses. This case is similar to the analysis of $RE$ on the Euclidean plane, so we do not need to apply our method, however in order to complete our analysis we have included it here. 
In this case the spherical angle $\theta_k = \pi/2$
and 
$\phi_k(t)=\phi_k(0)+\omega t$
for $k=1,2,3$.
When $\omega = 0$ we will show that it corresponds to a fixed point.

In order to get examples of $RE$ 
we have to impose a kind of triangle inequalities
for the mass ratio.
Studying the limit cases when
the inequality becomes an equality, we obtain an interesting result about some antipodal singularities, given when in the limit, two particles collide in antipodal position with the third particle. 

In Section \ref{sec5} appears the main ideas of the paper, here we study $RE$ on a rotating meridian for a generic potential. For a $RE$ we have to distinguish the shape variables and the configuration variables and the relationship between them
(the translation formulas). In this way, we first get the equations of motion for ``rigid rotators'' (the shape) on rotating meridians, and then using the translation formulas we obtain the respective $RE$. 

The map from the configuration to the shape is straightforward. Nevertheless the converse, that is, the mapping from shape to configuration requires some condition, which is provided by the angular momentum. Thereby we obtain the translation formulas.

 Using this approach we get the equilateral relative equilibria on a rotating meridian for a generic potential.
In some sense, this is 
similar to find first a central configuration and then the corresponding $RE$ for the Newtonian $n$--body problem on the plane. 

Sections \ref{sec6} and \ref{sec7}
are dedicated to the analysis of $RE$ on a rotating meridian for the cotangent potential. The results in these sections provide new families of Eulerian $RE$ for the positive curved three body problem.

In Section \ref{sec8}, since for the cotangent potential, the number of $RE$ 
depends of the mass ratio and the fixed angle between one pair of masses, we determine the number of $RE$ on a rotating meridian. 
We give examples with exactly two and six $RE$. We conjecture that for the cotangent potential, for any masses and  
fixed angle,
the maximum number of $RE$ is six.
Finally in Section \ref{sec9} we summarize our results and state some final remarks.


\section{Preliminaries and equations of motion}\label{prelims}
We utilize spherical coordinates to describe the three body problem on $\mathbb{S}^2$. First we introduce the notations that we will use along the paper and then we get the equations of motion.

\subsection{Notations}
The point $(X,Y,Z)$ on $\mathbb{S}^2$ with radius $R$ is represented by
the
spherical coordinates $(\theta,\phi)$,
\begin{equation*}
(X,Y,Z)=R(\sin\theta\cos\phi,\sin\theta\sin\phi,\cos\theta).
\end{equation*}
The chord length between the  points
$(X_i,Y_i,Z_i)$ and $(X_j,Y_j,Z_j)$ is given by
\begin{equation}
\label{defD}
\begin{split}
D_{ij}^2
&=(X_i-X_j)^2+(Y_i-Y_j)^2+(Z_i-Z_j)^2\\
&=2R^2\Big(1-\cos\theta_i\cos\theta_j-\sin\theta_i\sin\theta_j\cos(\phi_i-\phi_j)
		\Big).
\end{split}
\end{equation}
The minor arc length is the shorter length between the point 
$(X_i,Y_i,Z_i)$ and $(X_j,Y_j,Z_j)$ along the geodesic (the great circle)
connecting the two points.
The arc angle $\sigma_{ij}$ is the minor arc length divided by $R$
and is equal to the angle between the two points 
as seen from the origin of $\mathbb{S}^2$.
Therefore we have,
\begin{equation}
0\le \sigma_{ij}\le \pi.
\end{equation}

The relation between the chord length $D_{ij}$ and  the arc angle $\sigma_{ij}$ is given by
\begin{equation}
\label{relationDandSigma}
\e D_{ij}=\sin(\sigma_{ij}/2) \quad
\mbox{ with } \quad
\e =1/(2R).
\end{equation}

The equations (\ref{defD}) and (\ref{relationDandSigma}) yield
the fundamental relation
for the arc angle and the 
spherical coordinates.
\begin{equation}
\label{fundamentalrelation}
\cos\sigma_{ij}
=\cos\theta_i\cos\theta_j+\sin\theta_i\sin\theta_j\cos(\phi_i-\phi_j).
\end{equation}

\subsection{Equations of motion}
The Lagrangian for the three body problem on
$\mathbb{S}^2$
is given by
\begin{equation}
\label{theLagrangian}
\begin{split}
&L=K+V,\\
&K=R^2\sum_{k=1,2,3}\frac{m_k}{2}
		\left(\dot{\theta}_k^2+\sin^2(\theta)\dot{\phi}_k^2\right),\quad
V=\sum_{i<j}m_i m_j U(D_{ij}^2),\\
\end{split}
\end{equation}
where dot on symbols represents the time derivative.
Along the paper we will use the  notation for the derivative
$U'$ which is defined by
\begin{equation}
U'(D^2)=\frac{dU(D^2)}{d(D^2)}<0.
\end{equation}
The last inequality is the assumption we make to ensure that
the force between any two bodies is attractive.

When we need to specify a particular potential,
we use the cotangent potential and its derivative given by

\begin{equation}
U(D_{ij}^2)
=\frac{1-2\e^2D_{ij}^2}{\sqrt{D_{ij}^2(1-\e^2D_{ij}^2)}}
=\frac{\cos\sigma_{ij}}{R\sqrt{1-\cos^2(\sigma_{ij})}},
\end{equation}
\begin{equation}\label{deriv-pot}
U'(D^2)=-\frac{1}{2R^3\sin^3(\sigma)}.
\end{equation}

The equations of motion for 
$\phi_k$ and $\theta_k$ are
\begin{align}
p_{\phi_k}
&=\frac{\partial L}{\partial \dot{\phi}_k}
=R^2 m_k \sin^2(\theta_k)\, \dot{\phi}_k,\\
\frac{d}{dt}p_{\phi_k}
&=\frac{\partial L}{\partial \phi_k}
=2R^2\left(
\sum_{i\ne k}m_im_kU'(D_{ik}^2)
	\sin\theta_i\sin\theta_k\sin(\phi_k-\phi_i)
\right),
\end{align}
\begin{equation}
p_{\theta_k}=\frac{\partial L}{\partial \dot{\theta}_k}=R^2 m_k \dot{\theta}_k,\\
\end{equation}
and $dp_{\theta_k}/dt=\partial L/\partial \theta_k$ yields
\begin{equation}
\begin{split}
&R^2 m_k \ddot{\theta}_k
-R^2 m_k \sin(\theta_k)\cos(\theta_k)\dot{\phi}_k^2\\
&=2\sum_{i\ne k} m_i m_kU'(D_{ij}^2)
	\left(
	\sin\theta_k\cos\theta_i-\cos\theta_k\sin\theta_i\cos(\phi_i-\phi_k)
	\right).
\end{split}
\end{equation}

Since the Lagrangian is  invariant under the action of the group $SO(3)$,
the angular momentum 
\begin{equation}
c=(c_x, c_y, c_z)=\sum_k m_k (X_k,Y_k,Z_k)\times (\dot{X}_k, \dot{Y}_k, \dot{Z}_k)
\end{equation}
is a first integral.
The expression for each component is
\begin{align}
c_x&=R^2 \sum_k m_k \left(-\sin(\phi_k) \dot\theta_k
	-\sin(\theta_k)\cos(\theta_k)\cos(\phi_k) \dot\phi_k\right),
	\label{defCx}\\
c_y&=R^2 \sum_k m_k \left(\cos(\phi_k) \dot\theta_k
	-\sin(\theta_k)\cos(\theta_k)\sin(\phi_k) \dot\phi_k\right),
	\label{defCy}\\
c_z&=R^2 \sum_k m_k \sin^2(\theta_k) \dot \phi_k.
	\label{defCz}
\end{align}

It is well known, but worth mentioning, that the equations of motion and the conservation  of the angular momentum are independent of the choice of coordinates.
That is, if the three bodies satisfy the equations of motion for one 
spherical coordinate system, they will also satisfy the equations of motion for other
spherical coordinates system with different choices of the z-axis. The same is true for angular momentum.

\section{Relative equilibria and their properties}\label{sec3}
We start this section with
the definition of relative equilibrium.

\begin{definition}[Relative equilibrium]\label{def1}
A relative equilibrium on $\mathbb{S}^2$ is a solution 
of the equations of motion which satisfy
$\dot{\theta}_k=0$ and $\dot{\phi}_i-\dot{\phi}_j=0$
for all $k=1,2,3$ and all pair $(i,j)$
in a 
spherical coordinate system.
\end{definition}
This definition is coordinate independent.
The existence of such coordinate system is the condition for
the relative equilibrium.

We observe that by relation (\ref{fundamentalrelation}),
the arc angle $\sigma_{ij}$ between bodies $i$ and $j$
is constant.
Therefore the shape of the triangle with $m_k$ at the vertex is
independent of time.
We call it
a ``rigid rotator'', because it rotates 
as if it were a rigid body. 
In other words, a ``rigid rotator'' represents the shape defined by the three angles $\sigma_{ij}$.
The ``congruent triangles'' represent the same rigid rotator.
Fixing the $z$--axis, we can determine angles $\theta_k$,
and then we get the relative equilibrium (configuration).

By definition \ref{def1}, $\dot{\phi}_k=\omega$ are common for
$k=1,2,3$. Then
 by the conservation of $c_z$,
\begin{equation}
c_z=R^2\omega \sum_k m_k \sin^2(\theta_k)
	\label{Cz}
\end{equation}
$\omega$ is a constant.
Then, $\phi_k(t)=\phi_k(0)+\omega t$.

The other two components of the angular momentum which must be zero are
expressed by
\begin{align}
c_x&=-R^2 \omega \sum_k m_k \sin(\theta_k)\cos(\theta_k)\cos(\phi_k),	\label{Cx}\\
c_y&=-R^2 \omega \sum_k m_k 
	\sin(\theta_k)\cos(\theta_k)\sin(\phi_k).\label{Cy}
\end{align}
The reason is the following.
If $\omega\ne 0$, comparing the time $t=0$ and $t=\pi/\omega$,
we obtain $c_x(0)=-c_x(\pi/\omega)$.
The conservation of $c_x(t)$ yields $c_x=0$. The same is true for $c_y$.
If $\omega=0$, $c_x=c_y=0$ is obvious.
Therefore, the angular momentum for a relative equilibrium has the form
\begin{equation}
c=(0,0,c_z) \quad \mbox{ if } \quad \omega\ne 0.
\end{equation}
That is, the 
spherical coordinate system in the definition is one where the z-axis is the direction of the angular momentum vector.
Thus the relative equilibrium must satisfy

\begin{equation}
\label{CxCy3}
\sum_k m_k\sin(\theta_k)\cos(\theta_k) e_k=0 \quad 
\mbox{ if } \quad \omega\ne 0,
\end{equation}
for the two-dimensional unit vector
\begin{equation}
e_k = (\cos\phi_k,\sin\phi_k).
\end{equation}

The inner product of $e_k$ and the previous equation yields
\begin{eqnarray}
\label{cxcyInE3}
& \nonumber m_i\sin(\theta_i)\cos(\theta_i)\cos(\phi_k-\phi_i)
+ m_j\sin(\theta_j)\cos(\theta_j)\cos(\phi_k-\phi_j) \\
& + m_k\sin(\theta_k)\cos(\theta_k) =0,
\end{eqnarray}
for $(i,j,k)=(1,2,3),(2,3,1)$, and $(3,1,2)$.
From here on, we will use the cyclic notation 
when one equation contains three indexes $(i,j,k)$.
The outer product of $e_k$ and the equation \eqref{CxCy3} yields
\begin{equation}
\label{cxcyInE4}
m_i\sin(\theta_i)\cos(\theta_i)\sin(\phi_k-\phi_i)
+m_j\sin(\theta_j)\cos(\theta_j)\sin(\phi_k-\phi_j)
=0.
\end{equation}
We will utilise these expressions later.

Finally we can write the equations of motion for relative equilibria.
Since $p_{\phi_k}=R^2\omega m_k\sin^2(\theta_k)$ is constant,
the equations of motion for $\phi_k$ are reduced to
$\partial L/\partial \phi_k=0$, which are equivalent to
\begin{equation}
\label{relationsFromPhiII}
\begin{split}
&m_1 m_2 U'(D_{12}^2)\sin\theta_1 \sin\theta_2 \sin(\phi_1-\phi_2)\\
=&m_2 m_3 U'(D_{23}^2)\sin\theta_2 \sin\theta_3 \sin(\phi_2-\phi_3)\\
=&m_3 m_1 U'(D_{31}^2)\sin\theta_3 \sin\theta_1 \sin(\phi_3-\phi_1).
\end{split}
\end{equation}
Since $\ddot{\theta}_k=0$, 
the equations of motion for $\theta_k$ are
\begin{equation}
\label{eqForThetaII}
\begin{split}
-\omega^2 m_k &\sin\theta_k\cos\theta_k\\
=&2m_k m_iU'(D_{ki}^2)
	\Big(
		\sin\theta_k\cos\theta_i-\cos\theta_k\sin\theta_i\cos(\phi_k-\phi_i)
	\Big)\\
	&+2m_km_jU'(D_{kj}^2)
	\Big(
		\sin\theta_k\cos\theta_j-\cos\theta_k\sin\theta_j\cos(\phi_k-\phi_j)
	\Big).
\end{split}
\end{equation}

So far, the properties of relative equilibrium are clarified.
The axis of rotation (z-axis) along which the three bodies rotate is the direction of the angular momentum.
The equations of motion to be satisfied by the three bodies are now clear.
Before closing this section
we add a definition
and a comment.
\begin{definition}[Fixed point]
A fixed point is a relative equilibrium with $\omega=0$.
\end{definition}
Since there are no motion for a fixed point,
we can take any direction for the $z$-axis.
Therefore, only the shape defined by $\sigma_{ij}$
has meaning.

Comment for the orientation:
If one triangle with $\sigma_{ij}$ is a rigid rotator,
another triangle with the same $\sigma_{ij}$ 
and opposite orientation
is also a rigid rotator.
Because the map 
to the antipodal point of each body,
namely $\theta_k\to \pi-\theta_k$
and $\phi_k\to \phi_k+\pi$,}
maps the original triangle
 to the triangle with opposite orientation,
and the conditions (\ref{CxCy3}) and equations of motion
(\ref{relationsFromPhiII}--\ref{eqForThetaII})
are invariant for this map.


\section{Relative equilibria on the Equator for the cotangent potential}\label{sec4}
In this section, we will determine the relative equilibria on the equator
for the cotangent potential.

We take the $z$--axis as the line passing through
the North and South poles of $\mathbb{S}^2$.
Three bodies are on the Equator if $\theta_k=\pi/2$
for $k=1,2,3$.

The condition for $c_x=c_y=0$ in (\ref{Cx}--\ref{Cy}) and the equations of motion for $\theta_k$ (\ref{eqForThetaII})
are satisfied.
Now, since we know that for the cotangent potential $U'(D_{ij}^2)=-1/(2R^3 |\sin(\phi_i-\phi_j)|^3)$ (see equation \eqref{deriv-pot}). The remaining equations for $\phi_k$
in (\ref{relationsFromPhiII}) are given by
\begin{equation}
\label{eqForEquator}
\frac{m_1m_2\sin(\phi_1-\phi_2)}{|\sin(\phi_1-\phi_2)|^3}
=\frac{m_2m_3\sin(\phi_2-\phi_3)}{|\sin(\phi_2-\phi_3)|^3}
=\frac{m_3m_1\sin(\phi_3-\phi_1)}{|\sin(\phi_3-\phi_1)|^3}.
\end{equation}

Note that the angular velocity $\omega$ does not appear
anywhere.
This is because
the centrifugal force produced by the rotation around the z-axis
is cancelled by the constraint force
which maintains the bodies on the surface of $\mathbb{S}^2$.
Therefore a solution of equation (\ref{eqForEquator})
is the solution for any $\omega$, including $\omega=0$.
In the last case, relative equilibria on the Equator
are the fixed points.

We can assume without loss of generality that $\sin(\phi_1-\phi_2)>0$, then by equation (\ref{eqForEquator}),
$\sin(\phi_2-\phi_3)>0$ and $\sin(\phi_3-\phi_1)>0$
follows and 
 equation (\ref{eqForEquator}) is reduced to
\begin{equation}
\label{theSineTheorem}
\frac{\sin(\phi_1-\phi_2)}{\mu_3}
=\frac{\sin(\phi_2-\phi_3)}{\mu_1}
=\frac{\sin(\phi_3-\phi_1)}{\mu_2}
=\rho>0
\end{equation}
with
\begin{equation*}
\mu_k=\sqrt{m_im_j}.
\end{equation*}
We set the common value $\rho$.
The equation (\ref{theSineTheorem}) is just the sine theorem
for the triangle 
with perimeters length
$\mu_k$ and corresponding
outer angles  $\phi_i-\phi_j$.
See Figure \ref{figElementaryGeometry}.
\begin{figure}
   \centering
   \includegraphics[width=6cm]{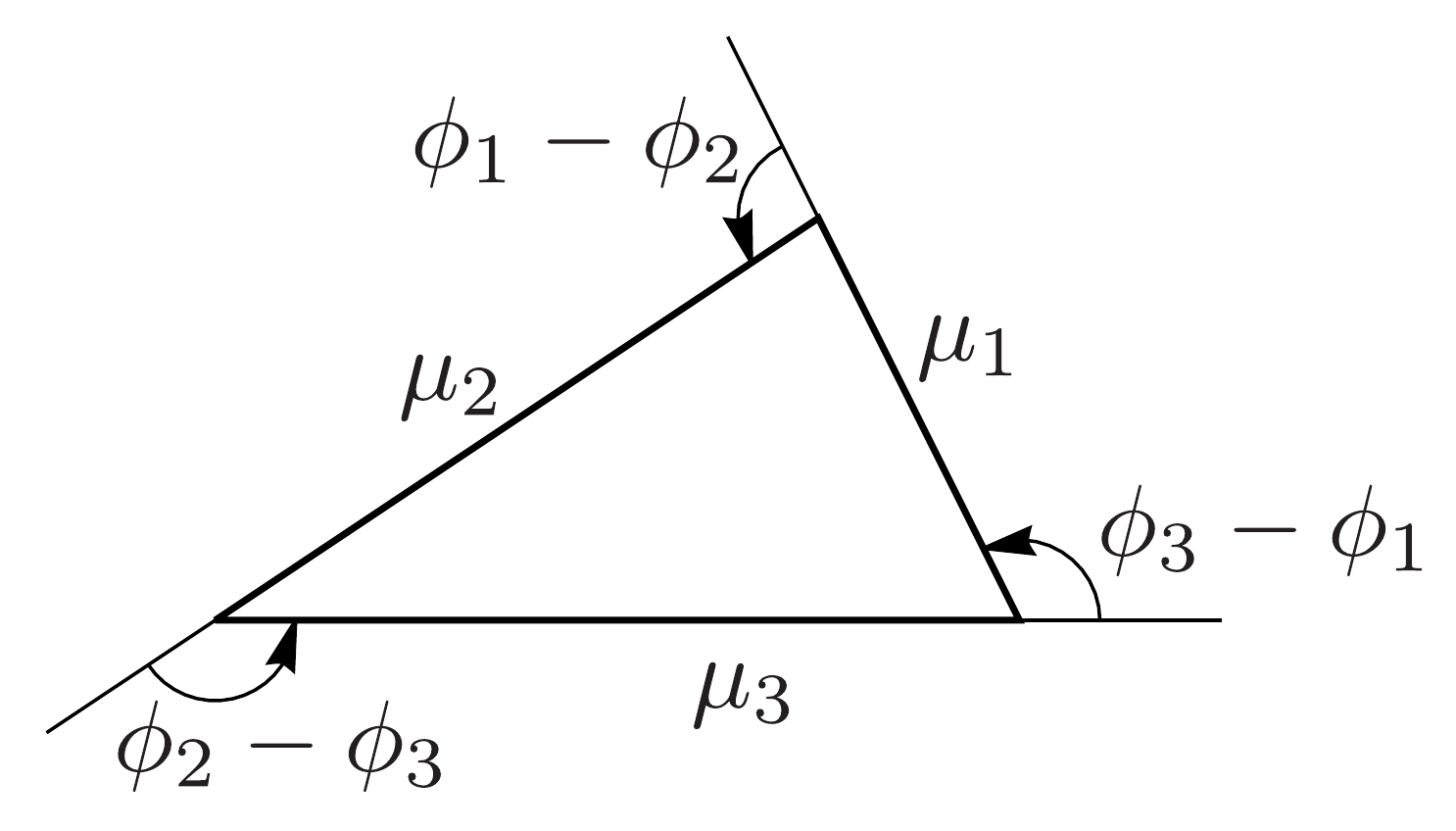} 
   \caption{The triangle with perimeters $\mu_k$
   and corresponding external angles $\phi_i-\phi_j$.}
   \label{figElementaryGeometry}
\end{figure}
Then, the cosine theorem yields
\begin{equation}
\label{theCosineTheorem}
\begin{split}
\cos(\phi_i-\phi_j)&=\frac{\mu_k^2-(\mu_i^2+\mu_j^2)}{2\mu_i\mu_j}.
\end{split}
\end{equation}
A direct algebraic derivation of this relation is given
in the Appendix \ref{directCosineTheorem}.
The common value $\rho$ is determined by
\begin{equation}
1
=\sin^2(\phi_3-\phi_1)+\cos^2(\phi_3-\phi_1)
=\rho^2\mu_2^2+\frac{\Big(\mu_2^2-(\mu_3^2+\mu_1^2)\Big)^2}{4\mu_3^2\mu_1^2},
\end{equation}
whose solution is
\begin{equation}
\label{solForK}
\rho
=\frac{\Big(
	(\mu_1+\mu_2+\mu_3)
	(\mu_1+\mu_2-\mu_3)
	(\mu_2+\mu_3-\mu_1)
	(\mu_3+\mu_1-\mu_2)
	\Big)^{1/2}}{2\mu_1\mu_2\mu_3}.
\end{equation}

The potential energy $-V$ for this equilibrium
is positive
\begin{equation}
\label{potentialEnergyForEquator}
\begin{split}
-V
&=-\sum_{ij}\frac{m_i m_j}{R\tan(\phi_i-\phi_j)}
=\frac{\sum_{ij}\mu_i^2\mu_j^2-\sum_k \mu_k^4}
	{Rk\mu_1\mu_2\mu_3}\\
%
&=\frac{\Big(
	(\mu_1+\mu_2+\mu_3)
	(\mu_1+\mu_2-\mu_3)
	(\mu_2+\mu_3-\mu_1)
	(\mu_3+\mu_1-\mu_2)
	\Big)^{1/2}}{R}\\
&>0.
\end{split}
\end{equation}

In order to get a triangle, the lengths 
$\mu_1,\mu_2,\mu_3$
 must satisfy the triangle inequalities,
\begin{equation*}
\mu_k < \mu_i+\mu_j.
\end{equation*}
These inequalities give a condition for the mass ratio to 
achieve a relative equilibrium (See Figure \ref{figSummaryOnEquator}).
\begin{equation*}
1 <\left(\frac{m_3}{m_1}\right)^{1/2}
+\left(\frac{m_3}{m_2}\right)^{1/2}\!\!\!,
1 <\left(\frac{m_1}{m_2}\right)^{1/2}
+\left(\frac{m_1}{m_3}\right)^{1/2}\!\!\!,
1 <\left(\frac{m_2}{m_3}\right)^{1/2}
+\left(\frac{m_2}{m_1}\right)^{1/2}\!\!\!.
\end{equation*}

\begin{figure}
   \centering
   \includegraphics[width=7cm]{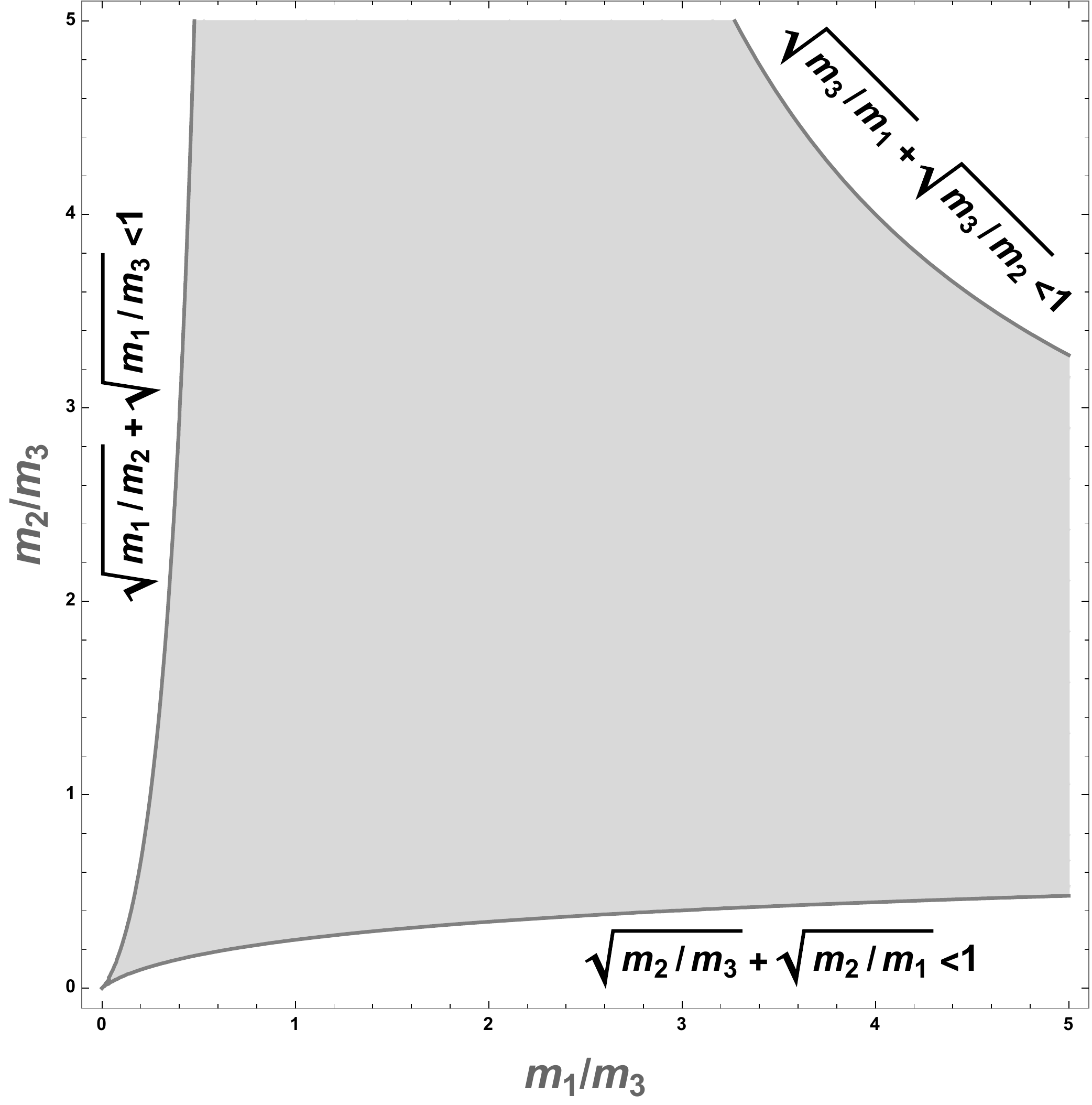} 
   \caption{The number of rigid rotator when three bodies are on the equator.
   The grey region has  $1$ solution.
   The three white regions 
   indicated by $(m_i/m_j)^{1/2}+(m_i/m_k)^{1/2}< 1$ do not have solutions.
   On the boundary between the grey and white regions given by
   $(m_i/m_j)^{1/2}+(m_i/m_k)^{1/2}= 1$,
   there are 
   no solutions. 
   }
   \label{figSummaryOnEquator}
\end{figure}
If the mass ratio satisfies these conditions,
the equations (\ref{theSineTheorem}),
(\ref{theCosineTheorem}) and
(\ref{solForK}) determine uniquely all $\phi_i-\phi_j$,
what is clearly shown in Figure \ref{figElementaryGeometry}.
In other words,
there is only one rigid rotator on the Equator
for given mass ratio.

The above inequalities bring us an interesting question.
What will happen if we take the limit
satisfying the equality?
For instance taking,
\begin{equation}
\label{forTheLimitingCase}
\left(\frac{m_3}{m_1}\right)^{1/2}
+\left(\frac{m_3}{m_2}\right)^{1/2}
\to 1.
\end{equation}
This corresponds to $\lim \mu_3 \to \mu_1+\mu_2$.
Then, $\phi_1-\phi_2\to0$,
$\phi_2-\phi_3\to \pi$ and $\phi_3-\phi_1\to \pi$.
Namely, both 
$m_1$ and $m_2$
approach to
the antipodal point of $m_3$.
However, the potential energy in (\ref{potentialEnergyForEquator})
 goes to zero.
Therefore, $m_1$ and $m_2$ can reach 
the antipodal point of $m_3$.
From the point of view of the equations of motion it is a singularity, but using the energy potential it is a normal point.
The derivative of the potential energy has the singularity. These kind of singularities are called {\it antipodal singularities}, they were first studied in 
\cite{Diacu-EPC2} for the three body problem on the sphere, where the particles move under the influence of the cotangent potential \eqref{partpot}. Moreover, with the above arguments it is possible to give a easiest proof of the main result in paper  \cite{Diacu-EPC2} in a 
general context.

\begin{remark}
Before to finish this section we would like to emphasizes that all the  above results are applicable when we have a repulsive force, in other words if we change the attractive potential $U = (1/R) \cot (\sigma)$
for the repulsive potential $U = - (1/R) \cot (\sigma)$, the same results hold.
\end{remark}


\section{Relative equilibria on a rotating meridian
for a generic potential}
\label{sec5}

In this section we develop a new method to study
relative equilibria on a rotating meridian
for a generic potential.

In order to analyze $RE$ on a rotating meridian,
it is convenient to set $\phi_k=\omega t$
for all $k=1,2,3$ and enlarge the range of $\theta_k$
to $-\pi\le \theta_k\le \pi$.

Since $\sin(\phi_i-\phi_j)=0$ for all pair $(i,j)$,
the equations of motion for $\phi_k$ are satisfied.
The condition for $c_x=c_y=0$ are simplified to
\begin{equation}
\label{CxCyForMeridian}
\sum_k m_k \sin(2\theta_k)=0.
\end{equation}
And the equations of motion for $\theta_k$ are
\begin{equation}
\label{eqThetaForMeridian0}
\begin{split}
\frac{\omega^2}{2}m_k\sin(2\theta_k)
&=-2\sum_{i\ne k}
	m_km_i\sin(\theta_k-\theta_i)U'(D_{ki}^2).\\
\end{split}
\end{equation}

The righthand side of the last equation contains only the difference
of $\theta_k$
and $\theta_i$, which describes the shape.
Therefore it is free from the choice of the $z$--axis.
On the other hand, the lefthand side depends on $\theta_k$,
which describes the configuration,
and depends on the choice of the $z$--axis.
The information for the
 translation formula between 
the shape variables $\theta_i-\theta_j$
and the configuration variables $\theta_k$
is contained in the condition (\ref{CxCyForMeridian}).

Now we will obtain the translation formula,
then we go back to solve equation (\ref{eqThetaForMeridian0}).

\subsection{Translation between 
the shape variables and the configuration variables}

In this subsection,
we develop the translation formula
between the shape variables $\theta_i-\theta_j$ 
and the configuration variable $\theta_k$.
The equation (\ref{CxCyForMeridian}) will play the key role
in the translation.
This subsection would be the most important contribution
to the theory of equilibria on rotating meridian.

Let $\theta_{21}=\theta_2-\theta_1$ and 
$\theta_{31}=\theta_3-\theta_1$.
Then the condition (\ref{CxCyForMeridian}) can be written as
\begin{equation}
\label{CxCyForTheta1}
\begin{split}
0
=&m_1\sin(2\theta_1)+m_2\sin(2\theta_1+2\theta_{21})
		+m_3\sin(2\theta_1+2\theta_{31})\\
=&\sin(2\theta_1)\Big(
		m_1+m_2\cos(2\theta_{21})
				+m_3\cos(2\theta_{31})
		\Big)\\
	&+\cos(2\theta_1)\Big(
		m_2\sin(2\theta_{21})
				+m_3\sin(2\theta_{31})
		\Big)\\
=&A\sin(2\theta_1+2\alpha),
\end{split}
\end{equation}
where
\begin{equation}
\label{defA}
\begin{split}
A
&=\left(\big(
		m_1+m_2\cos(2\theta_{21})
				+m_3\cos(2\theta_{31})
		\big)^2
	+\big(
		m_2\sin(2\theta_{21})
				+m_3\sin(2\theta_{31})
		\big)^2
\right)^{1/2}\\
&=
\left(
	\sum_k m_k^2
	+2\sum_{i<j} m_i m_j \cos\Big(2(\theta_j-\theta_i)\Big)
\right)^{1/2},
\end{split}
\end{equation}
and
\begin{equation}
\begin{split}
\cos(2\alpha)&=A^{-1}\left(
	m_1+m_2\cos\Big(2(\theta_2-\theta_1)\Big)
				+m_3\cos\Big(2(\theta_3-\theta_1)\Big)
	\right),\\
\sin(2\alpha)&=A^{-1}\left(
	m_2\sin\Big(2(\theta_2-\theta_1)\Big)
				+m_3\sin\Big(2(\theta_3-\theta_1)\Big)
	\right).
\end{split}
\end{equation}

In the following, we assume  
$A\ne 0$.
The case $A=0$ will be discussed in the last part
of this subsection.

The  solutions of (\ref{CxCyForTheta1})
are $\theta_1=-\alpha$, and  $\theta_1 = -\alpha+\pi/2$.
Namely, the solutions are
\begin{equation}
\begin{split}
\cos(2\theta_1)&=s A^{-1}\left(
	m_1+m_2\cos\Big(2(\theta_1-\theta_2)\Big)
				+m_3\cos\Big(2(\theta_1-\theta_3)\Big)
	\right),\\
\sin(2\theta_1)&=s A^{-1}\left(
	m_2\sin\Big(2(\theta_1-\theta_2)\Big)
				+m_3\sin\Big(2(\theta_1-\theta_3)\Big)
	\right).
\end{split}
\end{equation}
Where $s=\pm1$.
By introducing the complex number,
\begin{equation}
e^{2i\theta_1}
=sA^{-1}\left(
	m_1+m_2e^{2i(\theta_1-\theta_2)}
		+m_3e^{2i(\theta_1-\theta_3)}
	\right),
\end{equation}
the other angles are easily obtained,
\begin{equation}
\begin{split}
e^{2i\theta_2}
&=e^{2i\theta_1+2i(\theta_2-\theta_1)}
=sA^{-1}\left(
	m_1e^{2i(\theta_2-\theta_1)}
	+m_2
		+m_3e^{2i(\theta_2-\theta_3)}
	\right),\\
e^{2i\theta_3}
&=e^{2i\theta_1+2i(\theta_3-\theta_1)}
=sA^{-1}\left(
	m_1e^{2i(\theta_3-\theta_1)}
	+m_2 e^{2i(\theta_3-\theta_2)}
		+m_3
	\right).
\end{split}
\end{equation}

Thus, we obtain the translation formula
for $\sin(2\theta_k)$,
\begin{equation}
\label{sin2ThetaByDeltaTheta1}
\begin{split}
&\sin(2\theta_1)=sA^{-1}\Big(
				m_2\sin(2(\theta_1-\theta_2))
				+m_3\sin(2(\theta_1-\theta_3))
				\Big),\\
&\sin(2\theta_2)=sA^{-1}\Big(
				m_1\sin(2(\theta_2-\theta_1))
				+m_3\sin(2(\theta_2-\theta_3))
				\Big),\\
&\sin(2\theta_3)=sA^{-1}\Big(
				m_1\sin(2(\theta_3-\theta_1))
				+m_2\sin(2(\theta_3-\theta_2))
				\Big),
\end{split}
\end{equation}
and for $\cos(2\theta_k)$,
\begin{equation}
\label{cos2ThetaByDeltaTheta1}
\begin{split}
\cos(2\theta_1)
=s A^{-1}\left(
		m_1
		+m_2\cos\Big(2(\theta_1-\theta_2)\Big)
		+m_3\cos\Big(2(\theta_1-\theta_3)\Big)
		\right),\\
\cos(2\theta_2)
=s A^{-1}\left(
		m_1\cos\Big(2(\theta_2-\theta_1)\Big)
		+m_2
		+m_3\cos\Big(2(\theta_2-\theta_3)\Big)
		\right),\\
\cos(2\theta_3)
=s A^{-1}\left(
		m_1\cos\Big(2(\theta_3-\theta_1)\Big)
		+m_2\cos\Big(2(\theta_3-\theta_2)\Big)
		+m_3
		\right).
\end{split}
\end{equation}

With the above, we have obtained the translation formula between $\theta_k$ and 
$\theta_i-\theta_j$ if $A\ne 0$.

Now we consider the case $A=0$.
By equation (\ref{defA}),
$A=0$ is realized if and only if
\begin{equation}
\begin{split}
&m_1+m_2\cos(2(\theta_1-\theta_2))+m_3\cos(2(\theta_3-\theta_1))=0,\\
&m_2\sin(2(\theta_1-\theta_2))+m_3\sin(2(\theta_1-\theta_3))=0,
\end{split}
\end{equation}
whose solution is
\begin{equation}
\cos(2(\theta_i-\theta_j))
=\frac{m_k^2-(m_i^2+m_j^2)}{2m_im_j}.
\end{equation}
Actually, substituting the above expression into equation (\ref{defA})
yields $A=0$.
This can really happen.
A simple example is
$m_1=m_2=m_3$ and $\theta_i-\theta_j=\pm 2\pi/3$.
For these cases, by (\ref{CxCyForTheta1}),
the map $\theta_i-\theta_j$ to $\theta_k$ is
indefinite. 
Such situation corresponds to a fixed point.
It is exactly so, as we will see in
subsection (\ref{equilateralReOnMeridian}).

As we have shown above,
the translation from the shape variable $\theta_i-\theta_j$
to the configuration variable $\theta_k$ has two branches which
are represented by the value $s=\pm 1$.
The change $s \to -s$
corresponds to the change $\theta_k \to \theta_k+\pi/2$.
Namely, for the given shape and given mass,
there are two configurations,
which differ only by an overall angle of $90$ degrees.

The choice of $s$ is determined by the equations of motion.
Ahead in this paper we will describe this fact in detail.

Note that these formulas are derived only by the 
conditions $c_x=c_y=0$ with $\omega\ne 0$, therefore 
they only depend on the $SO(3)$ invariance of the system
and the definition of relative equilibria.

\subsection{Equations of motion for rigid rotators  on rotating meridian}

Using the translation between
$\theta_k$ and $\theta_i-\theta_j$
shown in the previous subsection,
the equations of motion (\ref{eqThetaForMeridian0}) for $A\ne 0$
can be written as
\begin{equation}
\label{eqThetaForMeridian}
\begin{split}
m_1m_2\left(s\frac{\omega^2}{2A}\sin\Big(2(\theta_1-\theta_2)\Big)
+2\sin(\theta_1-\theta_2)U'(D_{12}^2)
\right)\\
=m_2m_3\left(s\frac{\omega^2}{2A}\sin\Big(2(\theta_2-\theta_3)\Big)
+2\sin(\theta_2-\theta_3)U'(D_{23}^2)
\right)\\
=m_3m_1\left(s\frac{\omega^2}{2A}\sin\Big(2(\theta_3-\theta_1)\Big)
+2\sin(\theta_3-\theta_1)U'(D_{31}^2)
\right).
\end{split}
\end{equation}

These equations contain only the 
arc angle between the bodies.
Thus, we obtain the equations for rigid rotators
that are free from the choice of the $z$--axis.
Now the path that we will follow in or analysis is clear, we will use equation \eqref{eqThetaForMeridian} to find a rigid rotator, then we will apply the translation formulas \eqref{sin2ThetaByDeltaTheta1} and \eqref{cos2ThetaByDeltaTheta1} to obtain the corresponding relative equilibrium.

The role of the equation (\ref{eqThetaForMeridian})
will be more clear if we write them in the form,
\begin{equation}
\label{conditionForRigidRotatorI}
\begin{split}
s\frac{\omega^2}{2A}(G_{12}-G_{23})&=F_{12}-F_{23},\\
s\frac{\omega^2}{2A}(G_{23}-G_{31})&=F_{23}-F_{31},\\
s\frac{\omega^2}{2A}(G_{31}-G_{12})&=F_{31}-F_{12},
\end{split}
\end{equation}
where
\begin{equation}
\label{defOfFandG}
\begin{split}
F_{ij}&=-2m_im_j\sin(\theta_j-\theta_i)U'(D_{ij}^2),\\
G_{ij}&=m_im_j\sin\Big(2(\theta_j-\theta_i)\Big).
\end{split}
\end{equation}

Only two equations in (\ref{conditionForRigidRotatorI})
are independent. Without loss of generality 
we take the first and the last one.
There are four cases.

\begin{itemize}
\item[i)] Case 1.

For the case $(G_{12}-G_{23})(G_{31}-G_{12})\ne 0$.
Then the equations (\ref{conditionForRigidRotatorI}) is
equivalent to
\begin{equation}
\label{sAndOmega}
\begin{split}
\frac{F_{12}-F_{23}}{G_{12}-G_{23}}
=\frac{F_{31}-F_{12}}{G_{31}-G_{12}}
=s\frac{\omega^2}{2A}.
\end{split}
\end{equation}
The left hand side of this equality determines the rigid rotator, and the right hand side determines 
$s$ and $\omega^2$.

\item[ii)] Case 2.

For the case $(G_{12}-G_{23})=0$ and $(G_{31}-G_{12})\ne 0$,
the condition for a rigid rotator is $F_{12}-F_{23}=0$,
and $s$ and $\omega^2$ are determined by
\begin{equation}
\label{exceptionalCaseI}
\frac{F_{31}-F_{12}}{G_{31}-G_{12}}
=s\frac{\omega^2}{2A}.
\end{equation}

\item[iii)] Case 3.

Similarly,
the case $(G_{12}-G_{23})\ne 0$ and $(G_{31}-G_{12})= 0$,
the condition for a rigid rotator is $F_{31}-F_{12}=0$
and $s$ and $\omega^2$ are determined by
\begin{equation}
\label{exceptionalCaseII}
\frac{F_{12}-F_{23}}{G_{12}-G_{23}}
=s\frac{\omega^2}{2A}.
\end{equation}

\item[iv)] Case 4.

For the case $G_{12}=G_{23}=G_{31}$,
the condition for the rigid rotator is
\begin{equation}
\label{exceptionalCaseIII}
F_{12}=F_{23}=F_{31}.
\end{equation}
For this case, $s$ and $\omega$ are not determined.
Any $s$ and $\omega$ satisfies the equation of motion.
Therefore, it corresponds to a fixed point.

\end{itemize}

The details of the exceptional cases in 
(\ref{exceptionalCaseI}--\ref{exceptionalCaseIII})
for the cotangent potential
are shown in the appendix \ref{exceptionalCases}.

If a shape $\theta_i-\theta_j$ satisfies 
one of the conditions
(\ref{sAndOmega}),
 (\ref{exceptionalCaseI}), (\ref{exceptionalCaseII}) or
(\ref{exceptionalCaseIII})
then the shape is a rigid rotator.
If you change the potential $U$ (attractive) to $-U$ (repulsive),
the $F_{ij}-F_{jk}$ change their  sign,
and the equalities are satisfied by simply changing 
 $s \to -s$.
With the above we have proved the following proposition.
\begin{proposition}
If a shape for the
attractive potential $U$ is a rigid rotator, then
the same shape is a rigid rotator
for the repulsive potential $-U$.
The corresponding two configurations
differ only by an overall angle of $90$ degrees,
if not, it is a fixed point.
\end{proposition}

\subsection{
Equilateral relative equilibrium
on a rotating meridian for a generic potential}
\label{equilateralReOnMeridian}
The equations (\ref{eqThetaForMeridian}) have the following
simple solution. (This result was proved by the cotangent potential in \cite{zhu}).

\begin{proposition}
\label{equilateralRigidRotator}
The equilateral triangle
$\theta_1-\theta_2=\theta_2-\theta_3=\theta_3-\theta_1
=2\pi/3$ is
a rigid rotator on a rotating meridian
for any masses
and for the generic potential $U(D^2)$.
\end{proposition}

\begin{proof}
For $\theta_1-\theta_2=\theta_2-\theta_3=\theta_3-\theta_1
=2\pi/3$,
$D_{ij}=\sqrt{3}\, R$ for all $i,j$.
Therefore, $U'(D_{ij}^2)$ are common for all $i,j$.
Then the equations (\ref{eqThetaForMeridian}) are
\begin{equation}
\label{eqForEquilateralSolution}
\sqrt{3}\,\,m_im_j
\left(-s \frac{\omega^2}{4A}
	+U'
\right)
=
\mbox{ common  for all }
(i,j).
\end{equation}
Therefore, if
\begin{equation}
\begin{split}
s=-1,\quad
\omega^2=-4AU'
\end{split}
\end{equation}
the lefthand side of (\ref{eqForEquilateralSolution})
has a common value zero for all $(i,j)$.
Therefore, this shape is a rigid rotator.

The value of $A$ is given by
\begin{equation}
A
=\left(\sum m_k^2-\sum_{i<j} m_i m_j\right)^{1/2}
=\left(\frac{1}{2}\sum_{i<j} (m_i- m_j)^2\right)^{1/2}.
\end{equation}
Since $A\ne 0$ for non-equal masses case,
each value $\theta_k$ is determined by
the equations (\ref{sin2ThetaByDeltaTheta1}--\ref{cos2ThetaByDeltaTheta1}).
\begin{equation}
\label{cofigForEquilateralTriangle}
\begin{split}
\sin(2\theta_k)&=\frac{\sqrt{3}}{2A}(m_i-m_j),\\
\cos(2\theta_k)&=\frac{1}{A}\left(
					m_k-\frac{1}{2}(m_i+m_j)
				\right).
\end{split}
\end{equation}
We can check that they satisfy the original equations of motion
(\ref{eqThetaForMeridian0}).

If $m_1=m_2=m_3$ then $A=0$.
Therefore equations (\ref{eqThetaForMeridian})
and (\ref{sin2ThetaByDeltaTheta1}--\ref{cos2ThetaByDeltaTheta1})
are invalid.
However, the original
equations
of motion 
(\ref{eqThetaForMeridian0})
are satisfied by $\omega=0$ and
$\theta_1-\theta_2=\theta_2-\theta_3=\theta_3-\theta_1
=2\pi/3$.
This is a fixed point,
therefore individual value $\theta_k$ has no sense,
only the distance $\theta_i-\theta_j$ has meaning.

\end{proof}


\section{Relative equilibria on a rotating meridian
for the cotangent potential}
\label{sec6}
In this section we treat
relative equilibria on a rotating meridian
for the cotangent potential (\ref{partpot}).

By using \eqref{sAndOmega},
we can find a equation for $x=\theta_3-\theta_1$
for given masses and given $a=\theta_2-\theta_1$.
Then 
$\theta_3-\theta_2=x-a$. 

Here, we consider only the case in \eqref{sAndOmega}.
The exceptional cases in (\ref{exceptionalCaseI}--\ref{exceptionalCaseIII}) for the cotangent potential
are explicitly shown in the appendix \ref{exceptionalCases}.
These exceptional cases will be considered
when the value $a$ and the mass ratio are fixed.

Namely, in this section we assume,
\begin{equation}
\begin{split}
&(G_{12}-G_{23})(G_{31}-G_{12})\\
&=-m_1m_2m_3^2
\Big(\nu_1\sin(2a)-\sin(2(x-a))\Big)
\Big(\nu_2\sin(2a)+\sin(2x)\Big)\\
&\ne 0.
\end{split}
\end{equation}
Where,
\begin{equation}
\nu_1=\frac{m_1}{m_3},\qquad
\nu_2=\frac{m_2}{m_3}.
\end{equation}

It is not difficult to verify that the potential has four singular points
$\theta_3=\theta_1, \theta_2, \theta_1+\pi,$ and 
$\theta_2+\pi$.

Without loss of generality
we can assume $0<a<\pi$.
That is, $\theta_1<\theta_2<\theta_1+\pi<\theta_2+\pi$.
 
We name the regions I, II, III, and IV where the potential is regular
as follows,
\begin{equation}
\begin{cases}
\textrm{I:}&\theta_1<\theta_3<\theta_2,\\
\textrm{II:}&\theta_2<\theta_3<\theta_1+\pi,\\
\textrm{III:}&\theta_1+\pi<\theta_3<\theta_2+\pi,\\
\textrm{IV:}&\theta_2+\pi<\theta_3<\theta_1+2\pi.
\end{cases}
\end{equation}
See Table \ref{tableForTheSignv3},
where the sign of $\sin(x)$ and $\sin(x-a)$ are shown.
The sign of $\sin(a)$ is positive by the assumption.
It is convenient to use the sign function defined by
\begin{equation}
|\sin(x)|=\alpha \sin(x),\qquad
|\sin(x-a)|=\beta \sin(x-a).
\end{equation}
The signs of $\alpha$ and $\beta$ are also shown
in Table \ref{tableForTheSignv3}.

\begin{table}
  \centering 
  \caption{Sign of $\sin(x)$ and $\sin(x-a)$
  and their sign function $\alpha$, $\beta$
  defined by $|\sin(x)|=\alpha \sin(x)$
  and $|\sin(x-a)|=\beta \sin(x-a)$.}
  \label{tableForTheSignv3}
\begin{tabular}{c|c|c|c|c|c|c|c|c}
\hline
   $\theta_3$&$\theta_1$&I&$\theta_2$&II&$\theta_1+\pi$&III&$\theta_2+\pi$&IV\\
   \hline
$x=\theta_3-\theta_1$&$0$&&$\theta_2-\theta_1$&&$\pi$&&$\theta_2-\theta_1+\pi$&\\
$\sin(x)$&$0$&$+$&$+$&$+$&$0$&$-$&$-$&$-$\\
$\alpha$&&$1$&$1$&$1$&&$-1$&$-1$&$-1$\\
\hline
$\theta_3-\theta_2$&$\theta_1-\theta_2$&&$0$&&$\theta_1-\theta_2+\pi$&&$\pi$&\\
$\sin(x-a)$&$-$&$-$&$0$&$+$&$+$&$+$&$0$&$-$\\
$\beta$&$-1$&$-1$&&$1$&$1$&$1$&&$-1$\\
\hline
\end{tabular}
\end{table}

The equation 
\begin{equation}
\label{eqForGivenTheta21}
\frac{F_{12}-F_{23}}{G_{12}-G_{23}}
=\frac{F_{31}-F_{12}}{G_{31}-G_{12}}
\end{equation}
for cotangent potential is
\begin{equation*}
\frac{\displaystyle
	\frac{m_1m_2}{\sin^2(a)}-\frac{m_2m_3}{\beta\sin^2(x-a)}
}{m_1m_2\sin(2a)-m_2m_3\sin(2(x-a))}
=
\frac{\displaystyle
	\frac{m_3m_1}{\alpha\sin^2(x)}+\frac{m_1m_2}{\sin^2(a)}
}{m_3m_1\sin(2x)+m_1m_2\sin(2a)},
\end{equation*}
which can be reduced to
\begin{equation}
\label{defOfG}
f=\frac{g}{
	\sin^2(x)\sin^2(x-a)} = 0,
\end{equation}
where
\begin{equation}
\begin{split}
g
=&\alpha\beta\sin^2(x)\sin^2(x-a)
	\Big(\nu_1\sin(2x)+\nu_2\sin(2(x-a))\Big)\\
&-\sin^2(a)\Big(\alpha\sin^2(x)\sin(2x)
			-\beta\sin^2(x-a)\sin(2(x-a))\Big)\\
&-\sin^2(a)\sin(2a)\Big(
		\nu_2\alpha\sin^2(x)+\nu_1\beta\sin^2(x-a)
	\Big).
\end{split}
\end{equation}

Note that
$g$ is a homogeneous function of $\sin$ 
of degree $5$.
Indeed, consider the limit $R\to \infty$
with the arc length $r_{31}=xR$, $r_{21}=aR$ finite.
Then, the lowest order term of $g$ is $O(1/R^5)$.
Let us expand $g$ in the region II, that is $\alpha=\beta=1$
and let 
$r_{31}=(1+\lambda) r_{21}$
we obtain,
\begin{equation}
\label{RtoInfLimit}
\begin{split}
g=(r_{21}/R)^5
	\Big(
	&(m_1+m_2)\lambda^5
	+(3m_1+2m_2)\lambda^4
	+(3m_1+m_2)\lambda^3\\
	&-(m_2+3m_3)\lambda^2
	-(2m_2+3m_3)\lambda
	-(m_2+m_3)
	\Big)\\
&+O((r_{21}/R)^7).
\end{split}
\end{equation}
The above fifth degree polynomial in $\lambda$, is the Euler polynomial found for this author in 1767 as it should be (you can find it in the original Euler's paper \cite{Euler}, and in \cite{Hestenes} page 403,
 for a modern lecture we recommend \cite{Moeckel}).
In other words,
the lowest order term of $g=0$ in the limit $R\to \infty$ is the same as the Euler polynomial for the Euclidean plane. The above result is close related with the result in Bengochea et al. \cite{Bengochea}, where the authors found the way to continue relative equilibria in the Newtonian $n$--body problem to spaces of constant curvature.

Note that the rigid rotators are given by the zeros
of $f$, not of $g$.
For example, for the case $a=\phi_2-\phi_1=\pi/2$,
$g$ has the factor $\sin(2x)$,
\begin{equation}
g=
\frac{\sin(2x)}{8}
\Big(\alpha\beta(\nu_1-\nu_2)(1-\cos(4x))
	+4(\alpha-\beta)\cos(2x)
	-4(\alpha+\beta)
\Big).
\end{equation}
Therefore, $g$ has zero points at $x=n\pi/2$, $n=0,1, 2,3$.
But the denominator  is 
$\sin^2(2x)/4.$
Therefore, the equation (\ref{defOfG}) for this case is
\begin{equation}
\label{eqForTheta21EqPiDiv2}
f= \frac{1}{2\sin(2x)}
\Big(\alpha\beta(\nu_1-\nu_2)(1-\cos(4x))
	+4(\alpha-\beta)\cos(2x)
	-4(\alpha+\beta)
\Big)=0.
\end{equation}
Thus,
zeros of $\sin(2x)$ are not the zeros of $f$,
while the zeros in the parentheses of the above equation give us the rigid rotators.
We will count the number of rigid rotators for the cotangent potential \eqref{partpot} in the following section.


\section{Some equilibria on a rotating meridian for the cotangent potential}\label{sec7} In this section, we apply our general results to the cotangent potential \eqref{partpot} to show
some interesting relative equilibria on rotating meridian.

\subsection{Rigid rotator with $\theta_2-\theta_1=\pi/2$}
\label{secReForThetaEqPiDiv2}
For $a=\theta_2-\theta_1=\pi/2$, $x=\theta_3-\theta_1$,
the condition for the exceptional cases
$(G_{12}-G_{23})(G_{31}-G_{12})= 0$
yields $\sin(2x)^2=0$.
The solutions $x=0,\pi/2,\pi,3\pi/2$,
correspond to the places for antipodal points of $m_1$ or $m_2$.
Thus, these values of  $x$ do not correspond to a rigid rotator.
Therefore, only the equation (\ref{eqForTheta21EqPiDiv2})
gives rigid rotators.

Let $h(x)$ be the expression between parentheses in 
(\ref{eqForTheta21EqPiDiv2}),
then, as we have mentioned before, the equation for a rigid rotator is
\begin{equation}
\label{eqForTheta21EqPiDiv2H}
h(x)=\alpha\beta(\nu_1-\nu_2)(1-\cos(4x))
	+4(\alpha-\beta)\cos(2x)
	-4(\alpha+\beta)=0.
\end{equation}

Now we are in conditions to count the number of solutions for $\theta_2-\theta_1=\pi/2$.
\begin{itemize}
\item[i)] For the region I:
$0<x<\pi/2$, $\alpha=1$, $\beta=-1$.
\begin{equation}
\begin{split}
h(x)
&=2\Big(4\cos(2x)-(\nu_1-\nu_2)\sin^2(2x)
	\Big).
\end{split}
\end{equation}
If $\nu_1=\nu_2$, then $h(x)=8\cos(2x)=0$ 
has just one solution $x=\pi/4$ in this region.
If $\nu_1 \ne \nu_2$, then
$h(x)=0$ is equivalent to
\begin{equation}
\sin(2x)\tan(2x)=\frac{4}{\nu_1-\nu_2}.
\end{equation}
Since the left hand side in the last equation takes any value in 
$(-\infty, \infty)$
just once for $0<x<\pi/2$,
$h(x)$ has exactly one solution.

\item[ii)] For the region II:
$\pi/2<x<\pi$, $\alpha=\beta=1$.
\begin{equation}
h(x)
=-(\nu_1-\nu_2)\cos(4x)+(\nu_1-\nu_2)-8.
\end{equation}
If $\nu_1-\nu_2=0$, there is no solution.
If $\nu_1-\nu_2\ne 0$ then
\begin{equation}
h(x)=0
\leftrightarrow
\cos(4x)=1-\frac{8}{\nu_1-\nu_2}.
\end{equation}
Therefore, the number of solutions in the region II is
\begin{equation*}
\begin{cases}
0 	& \mbox{for } \nu_1-\nu_2 <4,\\
1	& \mbox{for } \nu_1-\nu_2 =4,\\
2	& \mbox{for } \nu_1-\nu_2 >4.
\end{cases}
\end{equation*}

\item[iii)] For the region III:
$\pi<x<3\pi/2$, $\alpha=-1$, $\beta=1$.
\begin{equation}
h(x)=-2\Big((\nu_1-\nu_2)\sin^2(2x)+4\cos(2x)\Big).
\end{equation}
If $\nu_1-\nu_2=0$ then $\cos(2x)=0$ has a simple zero
at $x=5\pi/4$.
If $\nu_1-\nu_2\ne 0$
then $h(x)=0$ is equivalent to
\begin{equation}
\sin(2x)\tan(2x)=-\frac{4}{\nu_1-\nu_2},
\end{equation}
Since the left hand side takes any value in 
$(-\infty, \infty$
just once in the interval $0<x<\pi/2$,
$h(x)$ has exactly one solution.

\item[iv)] For the region IV:
$3\pi/2<x<2\pi$, $\alpha=\beta=-1$.
\begin{equation}
h(x)=-(\nu_1-\nu_2)\cos(4x)+(\nu_1-\nu_2)+8.
\end{equation}
Therefore, for  $\nu_1-\nu_2>-4$ there is no solution,
for $\nu_1-\nu_2=-4$ there is just one solution,
and for $\nu_1-\nu_2<-4$ there are two solutions.
\end{itemize} 

Summarizing: for $\theta_2-\theta_1=\pi/2$
the number of solutions in each region and the total number of solutions
are shown in Table \ref{NumberOfSolutionsTheta21PiDiv2}.
\begin{table}[htbp]
\caption{Number of solutions for $\theta_2-\theta_1=\pi/2$
in each region and the total.}
\label{NumberOfSolutionsTheta21PiDiv2}
\begin{center}\begin{tabular}
{c|cccc|c}
 & I & II & III & IV & Total \\
 \hline
$\nu_1-\nu_2<-4$ 	& 1 & 0 & 1 & 2 & 4 \\
$\nu_1-\nu_2=-4$ 	& 1 & 0 & 1 & 1 & 3 \\
$-4<\nu_1-\nu_2<4$ 	& 1 & 0 & 1 & 0 & 2 \\
$\nu_1-\nu_2=4$ 	& 1 & 1 & 1 & 0 & 3 \\
$4<\nu_1-\nu_2$ 	& 1 & 2 & 1 & 0 & 4
\end{tabular} 
\end{center}
\end{table}

\subsection{Isosceles rigid rotator}
In this subsection,
isosceles triangles are tackled.
It will be shown that
isosceles rigid rotator always exist for $\nu_1=\nu_2$
with any $\theta_2-\theta_1$.
Moreover, if
$\nu_1\ne\nu_2$, then rigid rotator can exist
for special values of $\theta_2-\theta_1$.

For ``smaller'' isosceles solution,
let be $a=\theta_2-\theta_1$,
and $x=\theta_3-\theta_1=a/2$
in the region I, namely $\alpha=1$, $\beta=-1$.
Then equation (\ref{defOfG}) reduces to 
\begin{equation}
\begin{split}
f=(\nu_1-\nu_2)(4\cos^2(a)+4\cos(a)-1)\sin(a)
=0.
\end{split}
\end{equation}

Therefore,
if $\nu_1=\nu_2$ then
any $a=\theta_2-\theta_1$ has an isosceles solution.
And for the special value of $a$ which satisfies
$4\cos^2(a)+4\cos(a)-1=0$, that is for

\begin{equation}
\cos(a)=\frac{\sqrt{2}-1}{2},
\end{equation}
an  isosceles solution exists for any $\nu_1$, $\nu_2.$
For this solution,
\begin{align}
&\sin(a)=\sin(\theta_2-\theta_1)=\frac{1}{2}\sqrt{1+2\sqrt{2}},\label{isoSinA}\\
&\sin(a/2)=\sin(\theta_2-\theta_3)=\sin(\theta_3-\theta_1)
	=\frac{1}{2}\sqrt{3-\sqrt{2}},\label{isoSinAdiv2}\\
&\cos(2a)=\cos(2(\theta_2-\theta_1))=\frac{1-2\sqrt{2}}{2},\\
&\cos(a)=\cos(2(\theta_2-\theta_3))=\cos(2(\theta_3-\theta_1))
=\frac{\sqrt{2}-1}{2},
\end{align}

\begin{equation}
A=\left(
	m_1^2+m_2^2+m_3^2
	+(2\sqrt{2}-1)
	(m_2m_3+m_3m_1-m_1m_2)
	\right)^{1/2} \ne 0,
\end{equation}
and
\begin{equation}
\frac{F_{12}-F_{23}}{G_{12}-G_{23}}
=\frac{F_{31}-F_{12}}{G_{31}-G_{12}}
=\frac{8}{7}\sqrt{\frac{13+16\sqrt{2}}{7}}.
\end{equation}
Therefore, by equation (\ref{sAndOmega}),
\begin{equation}
\label{isoOmega}
s=1
\mbox{ and }
R^3\omega^2
=2A\,\,\frac{F_{12}-F_{23}}{G_{12}-G_{23}}
=\frac{16A}{7}\sqrt{\frac{13+16\sqrt{2}}{7}}.
\end{equation}
Finally,
\begin{equation}
\begin{split}
\sin(2\theta_1)&=A^{-1}\big(-m_2\sin(2a)-m_3\sin(a)\big),\\
\sin(2\theta_2)&=A^{-1}\big(m_1\sin(2a)+m_3\sin(a)\big),\\
\sin(2\theta_3)&=A^{-1}\sin(a)\,\, (m_1-m_2).
\end{split}
\end{equation}
As seen in the last equation,
the configuration $\theta_k$ depends on the mass ratio.

Two comments for the above solution.
\begin{enumerate}
\item Since both $\sin a$ and $\cos a$ are positive,
$\sin(2a)$ is positive.
Therefore $(G_{12}-G_{23})(G_{31}-G_{12})
=-m_1m_2m_3^2(\nu_1 \sin(2a)+\sin a)(\nu_2 \sin(2a)+\sin a)
\ne 0$.
Namely, the use of the equation (\ref{defOfG}) is valid.
\item The shape of this rigid rotator does not depend of the masses.
We verify this fact by substituting (\ref{isoSinA}--\ref{isoSinAdiv2}) and (\ref{isoOmega}), into equation (\ref{eqThetaForMeridian})
and checking that each line is zero.
Therefore, this rotator does not depend on the masses.
The same happens for the equilateral rigid rotator.
\end{enumerate}

For ``larger'' isosceles solution,
let $a=\theta_2-\theta_1$,
and $x=\theta_3-\theta_1=a/2+\pi$
in the region III namely $\alpha=-1$, $\beta=1$.
Then equation (\ref{defOfG}) reduces to 
\begin{equation}
\label{eqForLargeIsoscele}
f=-(\nu_1-\nu_2)(2\cos(a)+1)^2\sin(a)=0.
\end{equation}
Therefore,
if $\nu_1=\nu_2$ then
any $a=\theta_2-\theta_1$ has an isosceles solution.
And if $a=2\pi/3$,
the  isosceles (actually equilateral) solution
exists
for any $\nu_1$, $\nu_2$.
The existence
of this rotator was already shown in section \ref{equilateralReOnMeridian}.

\section{Number of equilibria on rotating meridian
for given masses and given $\theta_2-\theta_1$ 
for the cotangent potential}
\label{sec8}
As we have shown in the previous sections, for the cotangent potential \eqref{partpot}, 
the number of relative equilibria depends on the
mass ratio and the angle $\theta_2-\theta_1$.
So a natural question arises. 
What is the minimum and maximum number of relative equilibria?

The key equation to give the answer is the equation (\ref{defOfG}).
To avoid the exceptional cases in 
(\ref{exceptionalCaseI}--\ref{exceptionalCaseIII}), we just consider the case
\begin{equation}
\label{theConditionForCountTheNumber}
\nu_k \sin(2a) >1 \quad \mbox{ for }k=1,2.
\end{equation}

 Then the factor 
 $\Big(\nu_1\sin(2a)-\sin(2(x-a))\Big)\Big(\sin(2x)+\nu_2\sin(2a)\Big)$
in $(G_{12}-G_{23})(G_{31}-G_{12})$ cannot be zero.

Although $f$ in (\ref{defOfG}) has denominator $\sin(x)\sin(x-a)=0$,
these zeros are on the boundaries of the regions I, II, III and IV.
A direct calculation shows that $g$ at the boundary is proportional to $\cos(a)\sin^5(a)$.
So, $g$ is not zero at the boundary if $\sin(2 a) \ne 0$. 
Therefore, the zeros of $f$ are zeros of $g$
if the condition (\ref{theConditionForCountTheNumber}) is satisfied.
 
Even though $g$ has step functions $\alpha$ and $\beta$,
 they always appear in the form
 $\alpha\sin^2(x)$ and $\beta\sin^2(x-a)$,
 therefore $g$ is continuous on the boundary of the regions I, II, III and IV.
 
 For the case studied in subsection \ref{secReForThetaEqPiDiv2},
 the minimum number of relative equilibria is two.
 However $a=\theta_2-\theta_1=\pi/2$  is a very special value,
 because $\sin(2a)=0$ for $a=\pi/2$.
 
 To see that there are at least one zero of $g$ in the region I 
 under the condition (\ref{theConditionForCountTheNumber}),
 note that
 \begin{equation}
 g(0)=-2\beta(\mu_1+1)\cos(a)\sin(a)^5,
 \quad g(a)=-2\alpha(\mu_2+1)\cos(a)\sin(a)^5
 \end{equation}
 and $\alpha\beta=-1$ for the region I.
 Therefore, $g(0)g(a)<0$, namely,
 there are at least one zero of $g$ in the region I.
 The same is true for the region III,
 and since $g(\pi)=g(0)$ and $g(\pi+a)=g(a)$, 
we obtain that at least two relative equilibria
 exist if (\ref{theConditionForCountTheNumber})
is satisfied.
 
The Figure \ref{figGforLargeMuPiDiv4and6} shows
an example with two relative equilibria,
$a=\pi/4$ and $\nu_1=3, \nu_2=2$
that satisfies the condition 
(\ref{theConditionForCountTheNumber}).
 
In order to find the maximum number of relative equilibria,
we saw in section \ref{secReForThetaEqPiDiv2}
an example with four relative equilibria.
However, there is an example with
$a=\pi/6$ and $\nu_1=3, \nu_2=2$,
which has six relative equilibria. This will be shown in the following subsection.

\subsection{Rigid rotator with $\theta_2-\theta_1=\pi/6$,\,\, $\nu_1=3,\,\,\nu_2=2$}
Here we show that it is possible to have six rigid rotators 
for $\theta_2-\theta_1=\pi/6$, $\nu_1=3$, $\nu_2=2$.
Since these parameters satisfy the condition (\ref{theConditionForCountTheNumber}),
it is enough to count zeros of the function $g(x)$.

In Figure \ref{figGforLargeMuPiDiv4and6},
we plot the graphic of $g(x)$.
We can see that there are six zeros.
One zero in the region I and III, and two zeros for each region II and IV. 

The existence of at least six zeros is shown in the following.

We have already shown the existence of at least one zero in the region I and III.
\begin{figure}
   \centering
   \includegraphics[width=5.5cm]{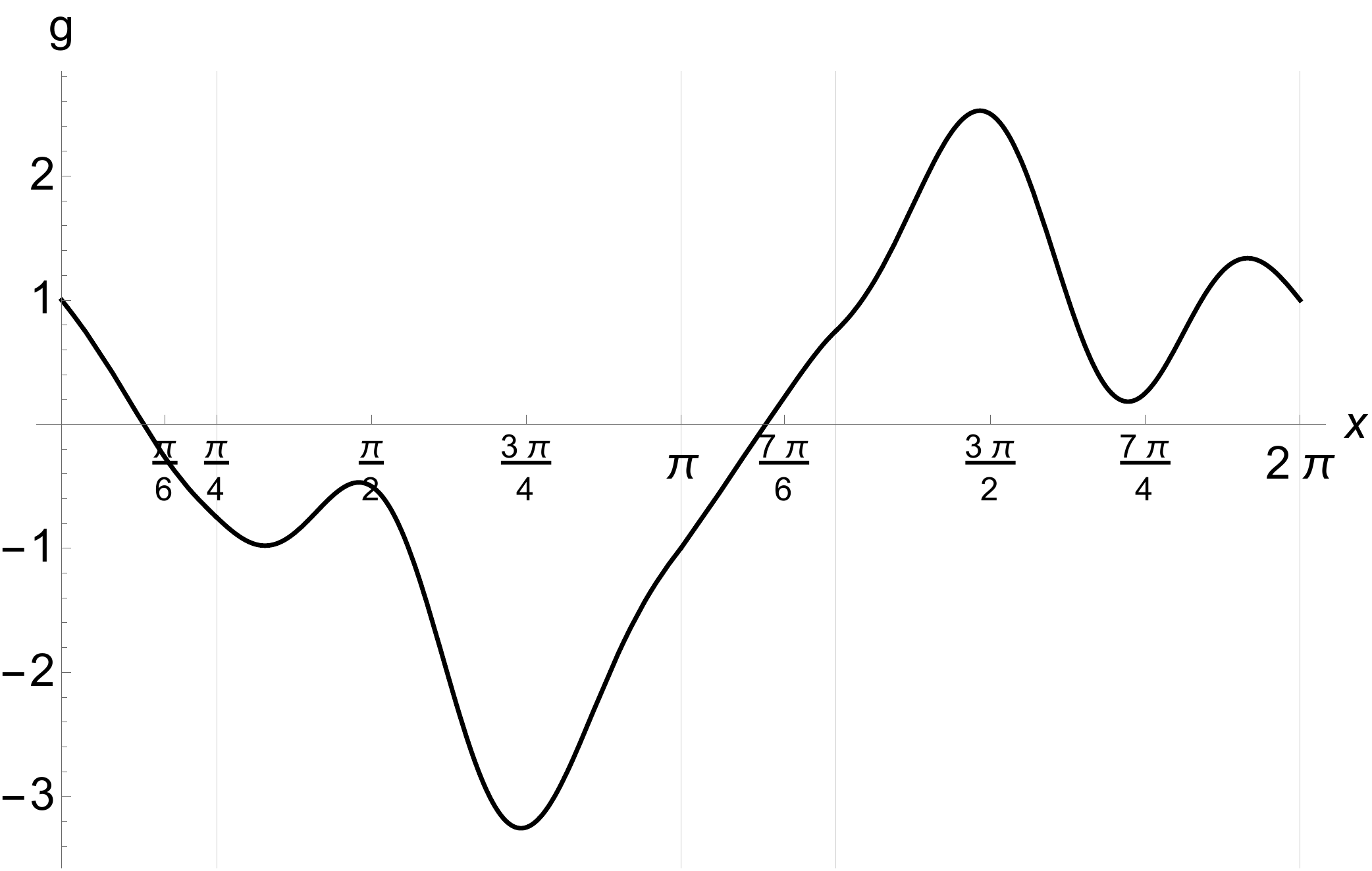}
   \includegraphics[width=6cm]{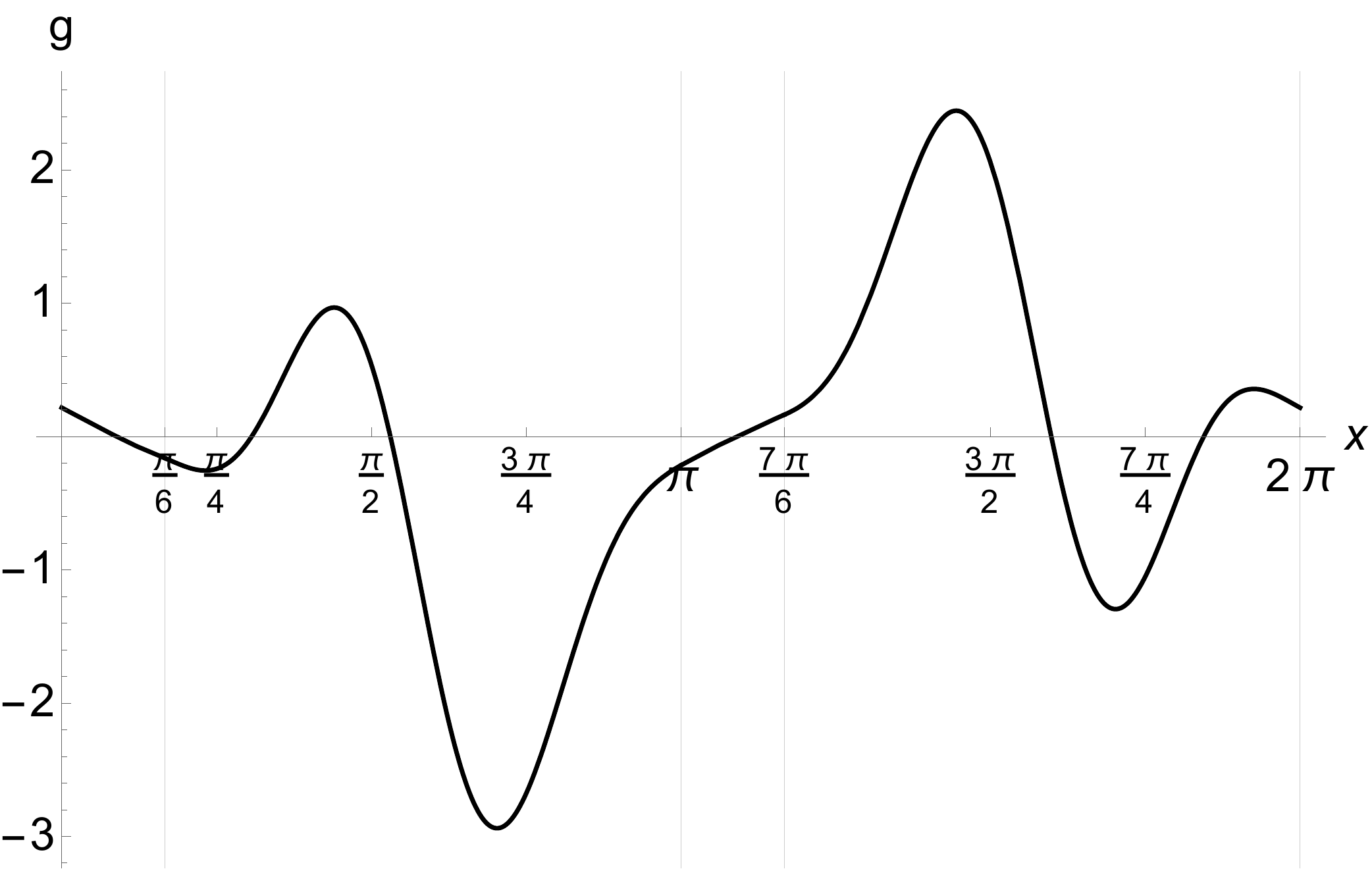} 
   \caption{The function $g$ for $\theta_2-\theta_1=\pi/4$ (left),
   and for $\theta_2-\theta_1=\pi/6$ (right)
   with $\nu_1=3$, $\nu_2=2$.
   The grey vertical lines represent the boundary of the region I, II, III, IV.
   }
   \label{figGforLargeMuPiDiv4and6}
\end{figure}

For the region II,
since
$g(\pi/6)=-3\sqrt{3}/32<0$,
$g(\pi/2)=5\sqrt{3}/16>0$,
$g(\pi)=-\sqrt{3}/8<0$,
there are at least two zeros.

For the region IV,
since
$g(7\pi/6)=3\sqrt{3}/32>0$,
$g(7\pi/4)=-5(5+\sqrt{3})/32<0$,
$g(2\pi)=\sqrt{3}/8>0$,
there are at least two zeros.

With all the above we have proved that
there are at least six rigid rotators 
if the condition (\ref{theConditionForCountTheNumber})
is satisfied.

The six rigid rotators (shapes)
and the corresponding six relative equilibria (configurations)
are shown in Figure \ref{figM321v5DeltaPiDiv6Fig6AllShape}
and Figure \ref{figM321v5DeltaPiDiv6Configurations}
respectively.

We can observe in Figure \ref{figM321v5DeltaPiDiv6Fig6AllShape}
that the position of the mass 3 is almost
symmetrically distributed for the line connecting the middle point of $m_1$ and $m_2$ and the centre of $\mathbb{S}^2$.
The symmetry is perfect if 
$m_1=m_2.$

\begin{figure}   \centering
   \includegraphics[width=4cm]{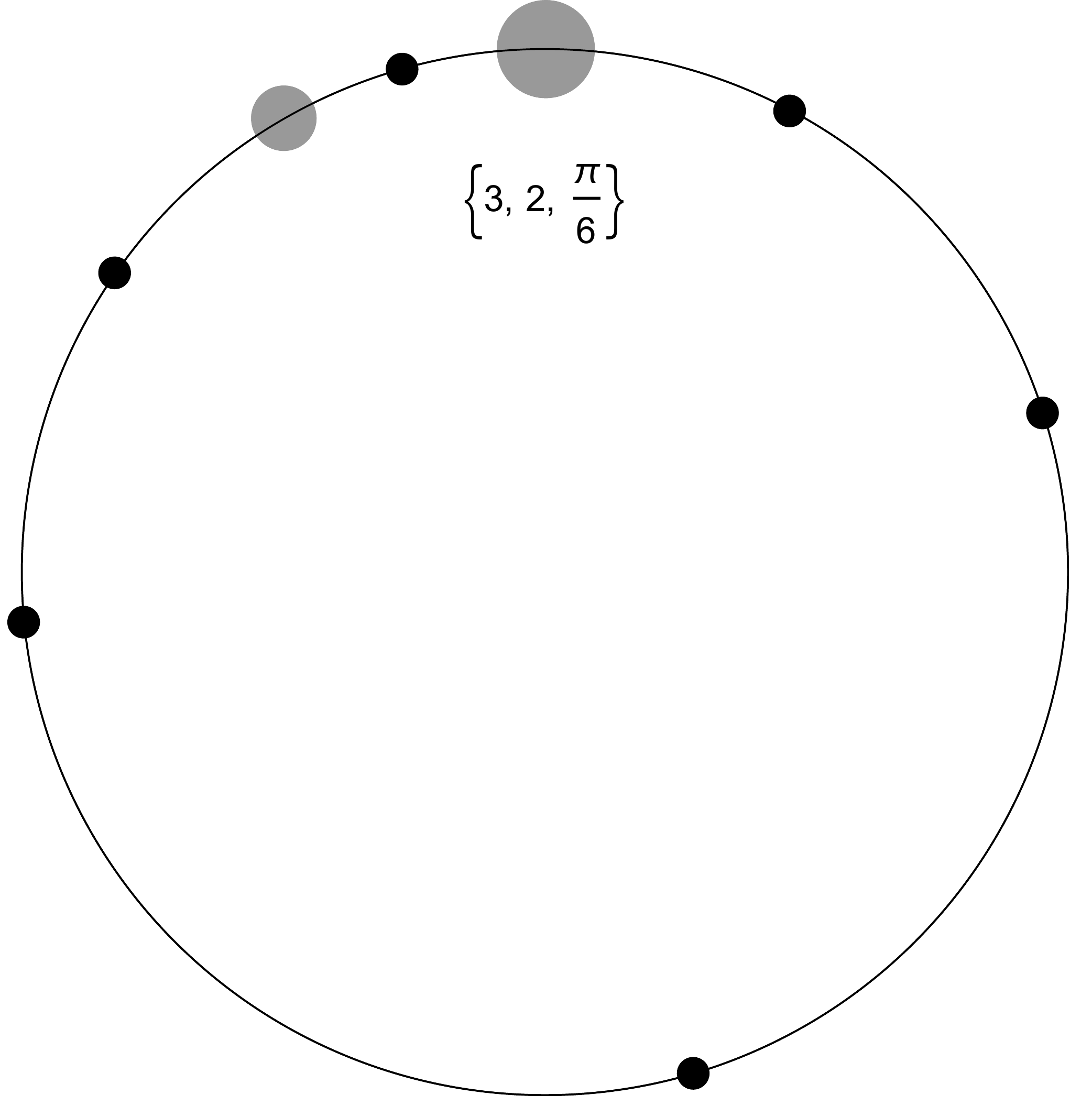} 
   \caption{Six rigid rotator (shape) with
   $\nu_1=3,\nu_2=2$,
   $\theta_2-\theta_1=\pi/6$ in one picture.
   The grey circles represent the bodies 1 (biggest) and 2,
   whose difference is fixed to $\theta_2-\theta_1=\pi/6$.
   The six black circles represent the position of the mass 3
   for six rigid rotator.
   }
   \label{figM321v5DeltaPiDiv6Fig6AllShape}
\end{figure}

\begin{figure}
   \centering
   \includegraphics[width=4cm]{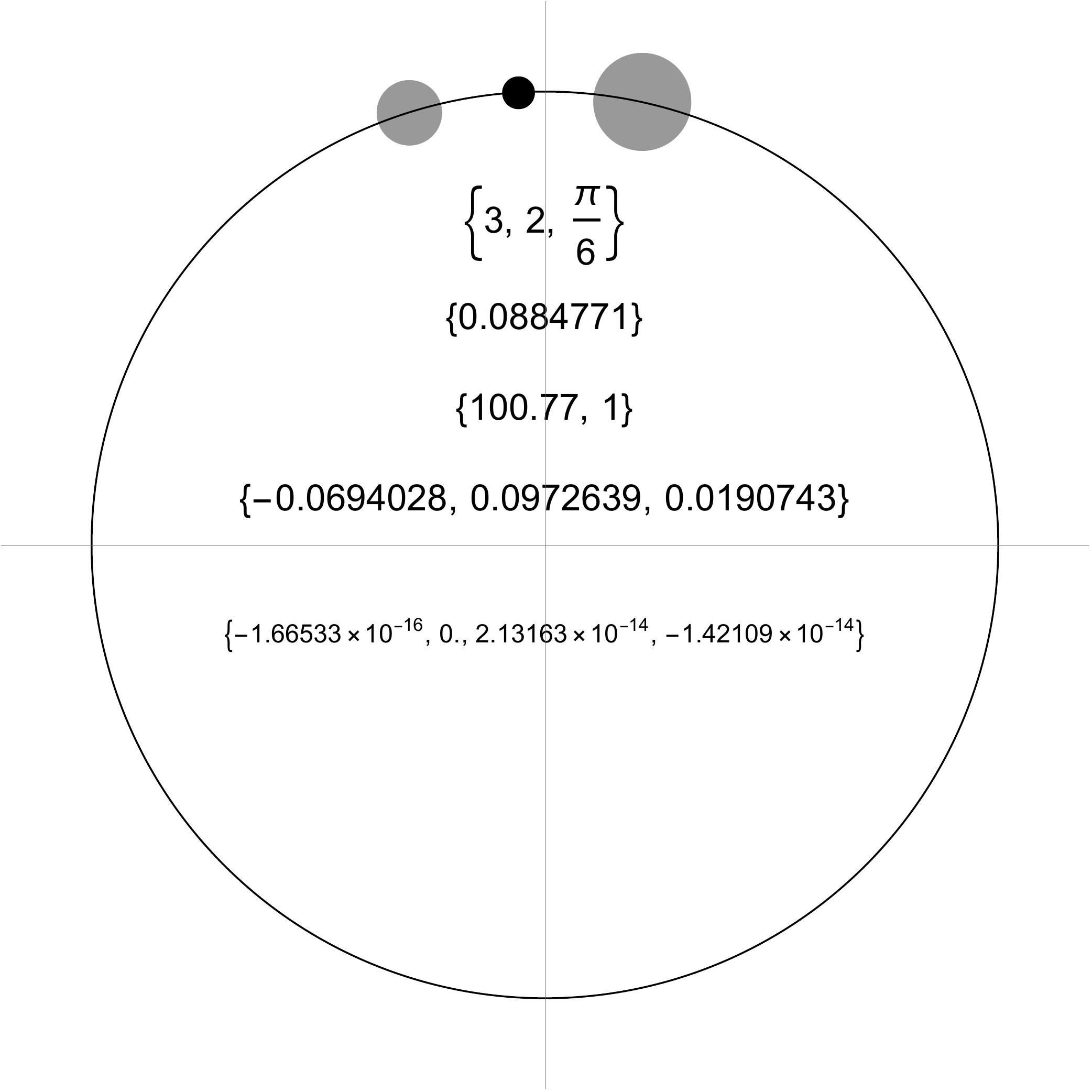}%
   \includegraphics[width=4cm]{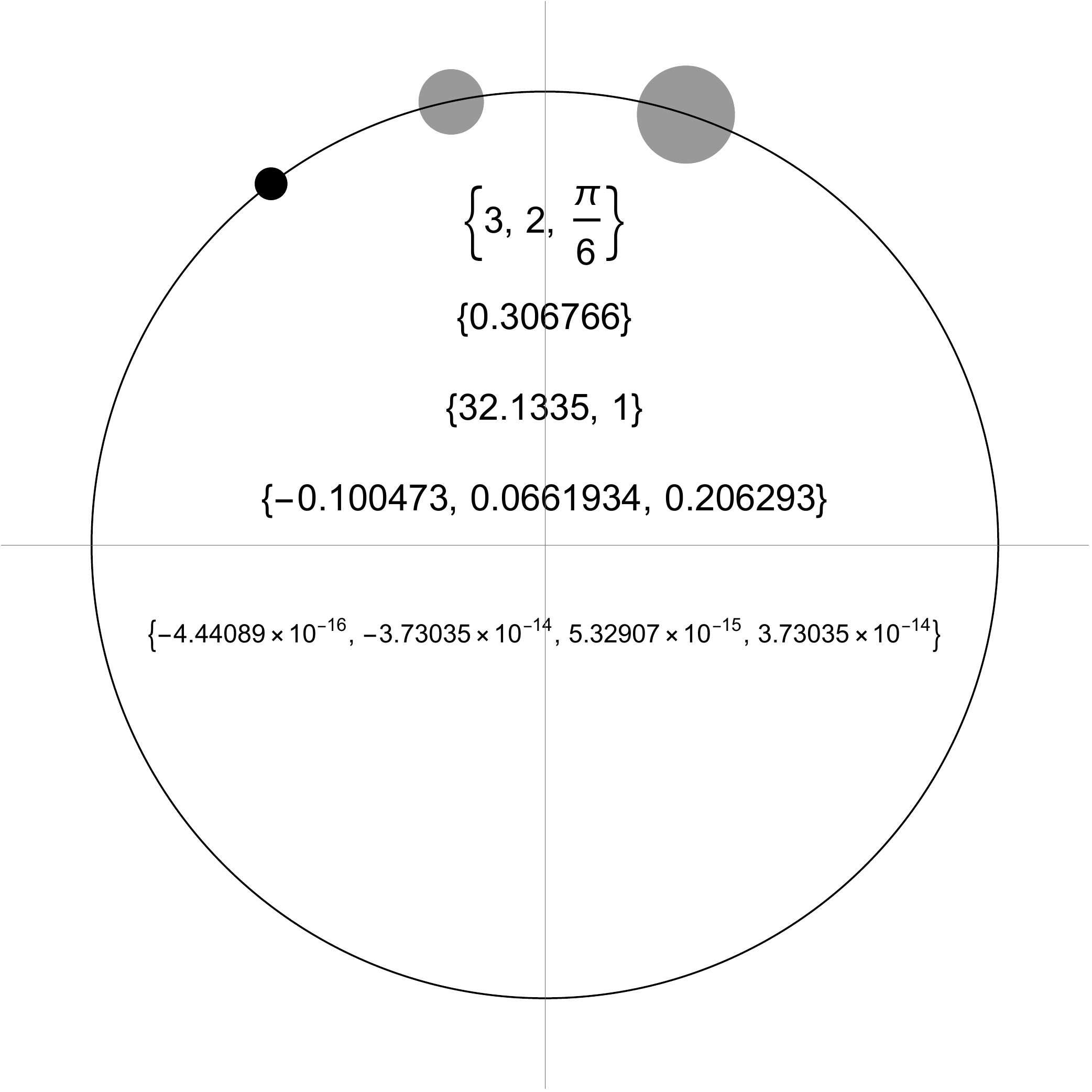}%
   \includegraphics[width=4cm]{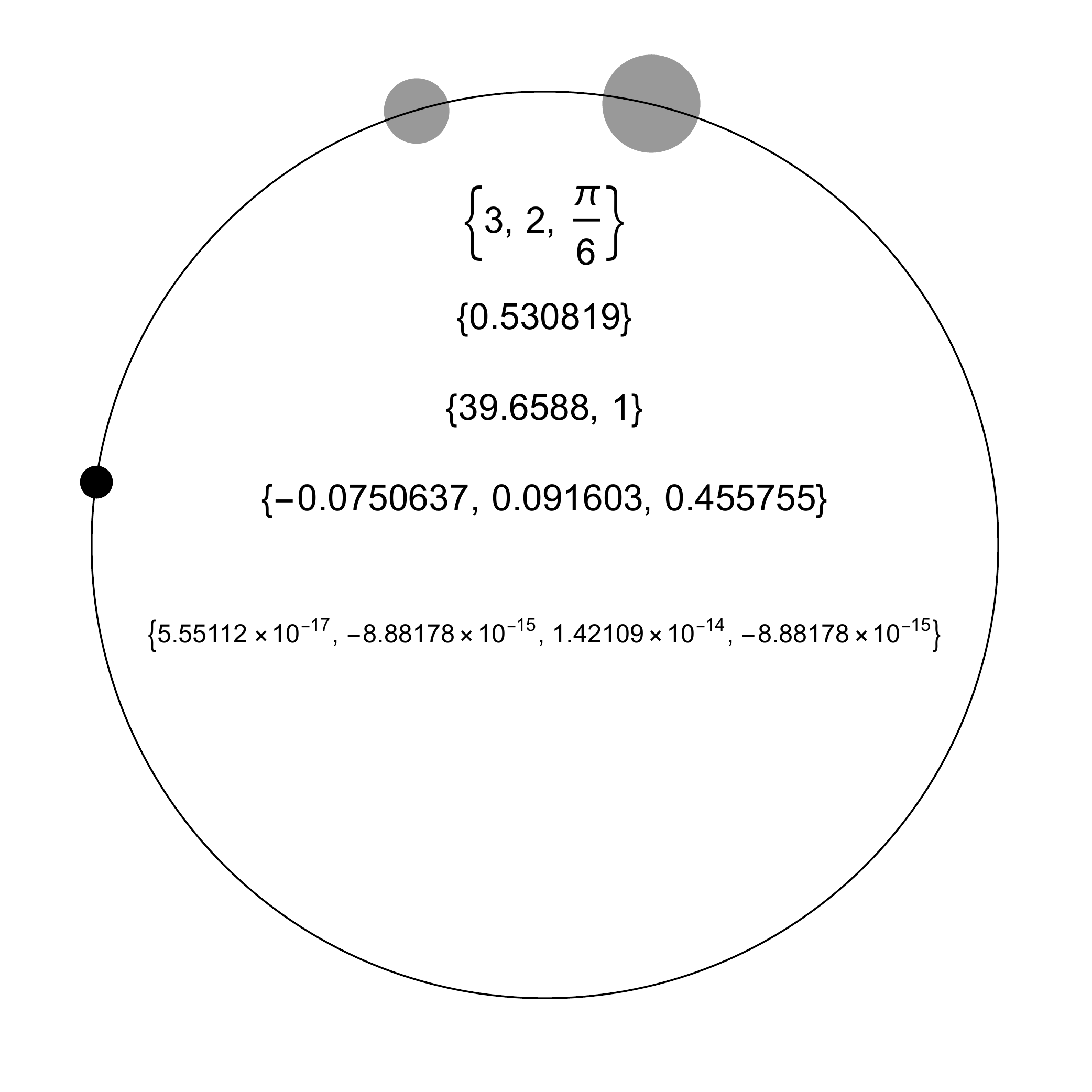}
   \includegraphics[width=4cm]{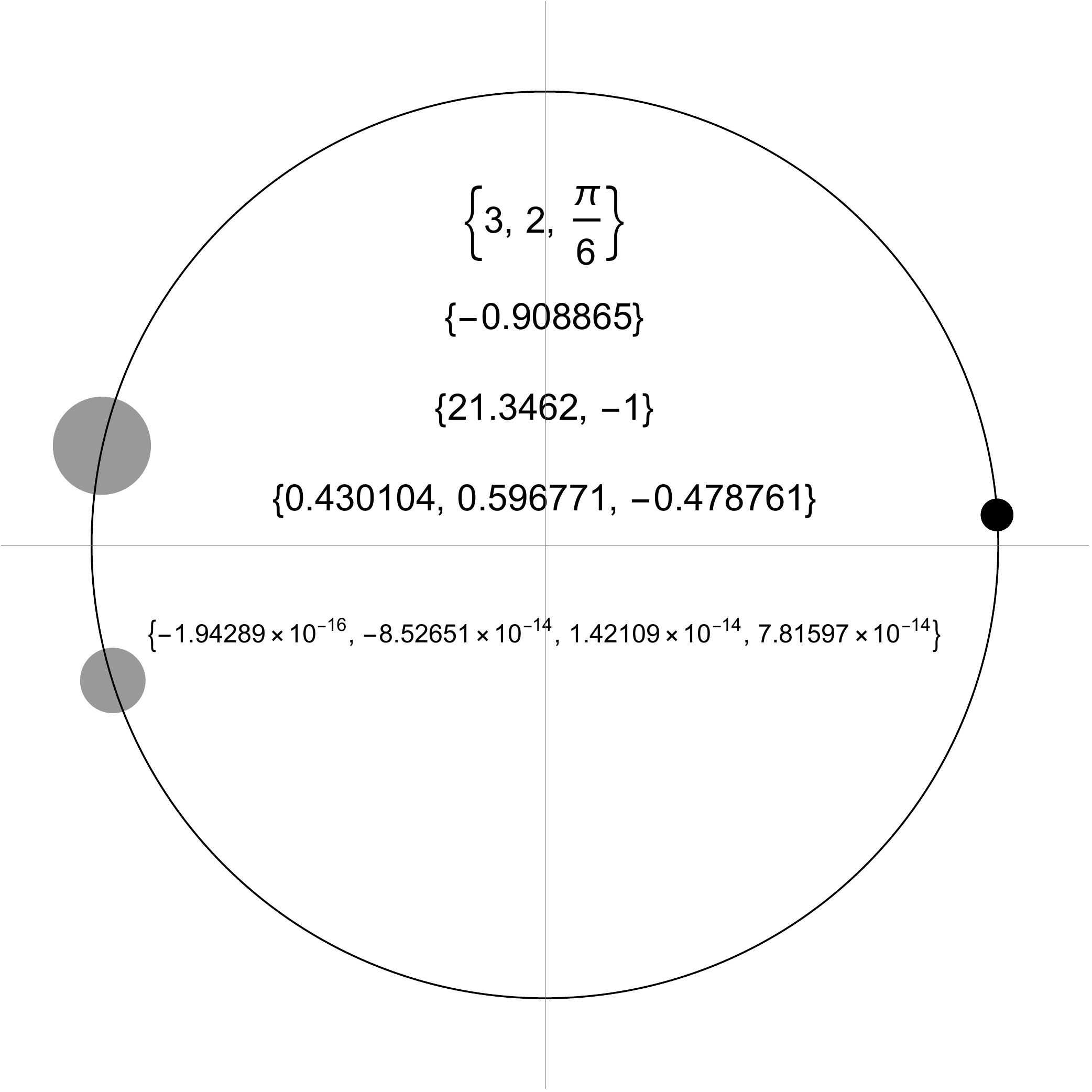}%
   \includegraphics[width=4cm]{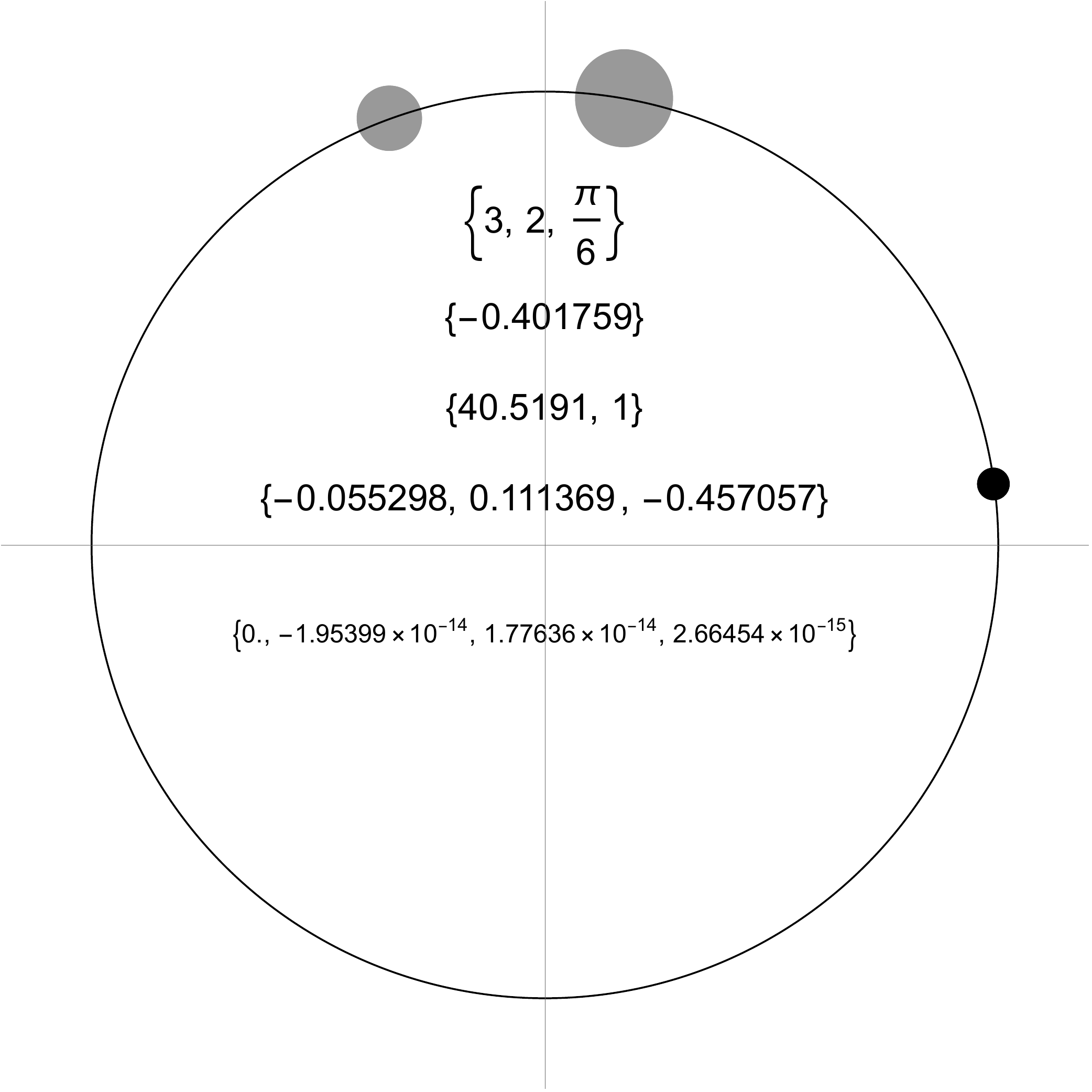}%
   \includegraphics[width=4cm]{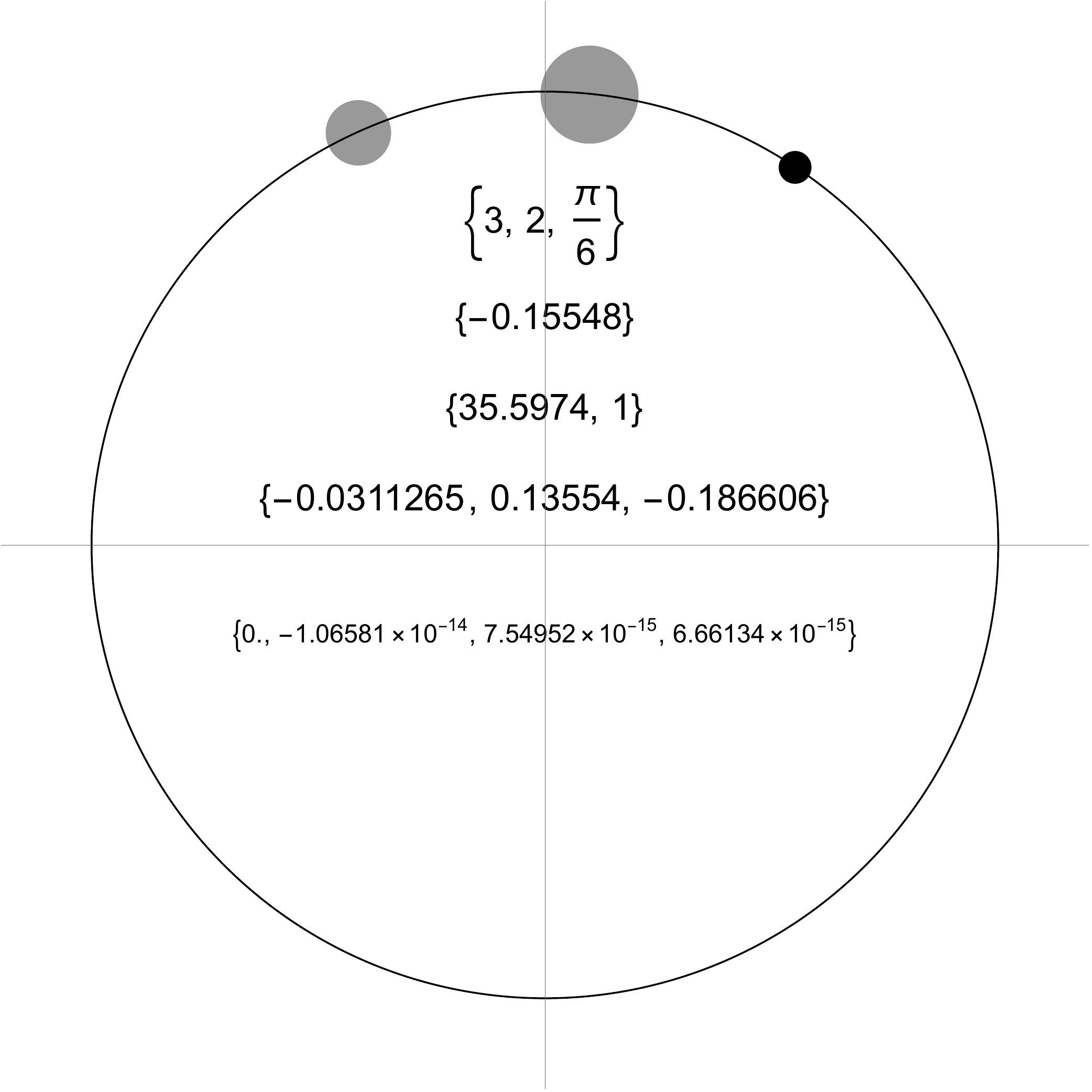}%
   \caption{Six relative equilibria (configurations) with
   $\nu_1=3,\nu_2=2$,
   $\theta_2-\theta_1=\pi/6$.
   The first line in each picture represents 
   $\{\nu_1,\nu_2,\theta_2-\theta_1\}$.
   The second, third, fourth line represents
   $(\theta_3-\theta_1)/\pi$,
   $\{R^3\omega^2,s\}$,
   and $\{\theta_1/\pi, \theta_2/\pi,\theta_3/\pi\}$.
   The fifth line shows
   the numerical difference between the lefthand and righthand side of the original equations of motion (\ref{CxCyForMeridian}) and (\ref{eqThetaForMeridian0}).
   }
   \label{figM321v5DeltaPiDiv6Configurations}
\end{figure}

Based on numerical experiments we state the following conjecture.

\begin{conjecture}
For the curved three body problem on $\mathbb{S}^2$ (with cotangent potential) the maximum number of collinear relative equilibria on a rotating meridian is 6. 
\end{conjecture}


\section{Conclusions and final remarks}\label{sec9}
We successfully derive the conditions for general rigid rotators, using this technique we obtain some specific rigid rotators and from here we 
find the corresponding relative equilibria
on $\mathbb{S}^2$.

When we have three bodies on the equator moving under the cotangent potential, 
the condition is given by equation (\ref{theSineTheorem}).
As shown clearly in Figure \ref{figElementaryGeometry},
the ratios of $\mu_k=\sqrt{m_i m_j}$ uniquely determine
$\phi_i-\phi_j$.
Since a rigid rotator on the Equator can rotate with any angular velocity $\omega$ including $\omega=0$,
they are also fixed points.
The choice of the $z$--axis is obvious.
Therefore,
the rigid rotator (shape) is directly connected to
the relative equilibrium (configuration). This fact is closely related with the role of the central configurations to find the relative equilibria in the Newtonian $n$--body problem.

In this paper, we did not mention rigid rotators on geodesics that are neither equatorial nor meridian. For the
bodies on such geodesics,
the centrifugal force for each particle
is on the plane of the meridian which passes through the body.
Therefore, the force cannot contribute to get the balance of
the gravitational forces between bodies, so
gravitational forces must balance by themselves,
and the angular velocity must be zero.
Therefore, the configuration must be a fixed point,
which is the same as for the fixed point on the Equator.

For the bodies on a rotating meridian,
the conditions for a rigid rotator are given by equations
(\ref{eqThetaForMeridian})
or one of (\ref{sAndOmega}--\ref{exceptionalCaseIII})
for generic potential, 
or (\ref{defOfG}) with exceptional cases in the appendix
\ref{exceptionalCases}
for the cotangent potential.

The equilateral  rigid rotators 
on a rotating meridian for arbitrary mass ratio
in a generic potential
exists, which is clearly shown by equation (\ref{eqThetaForMeridian}).
The $z$--axis is the rotation axis, and the configuration $\theta_k$
is given 
by equation (\ref{cofigForEquilateralTriangle}).
From here, we get that the relative equilibria (configuration of $\theta_k$) depend on the mass ratio.
On the other hand, the shape (rigid rotator) is
just one, equilateral triangle on the meridian.

The equation (\ref{defOfG}), especially $g=0$
for cotangent potential
has exactly the same limit for $R\to \infty$
as the equation for the 
Eulerian relative equilibria on the Euclidean plane,
as it should be. 
The equation (\ref{defOfG}) is useful
to find rigid rotators on meridians for given mass ratio and given
$\theta_2-\theta_1$.
In the paper we have found many rigid rotators, in particular we have showed that
isosceles rigid rotators exist.
For $m_1=m_2$ with any difference $\theta_2-\theta_1$, we found two isosceles rigid rotators where the mass $m_3$ is at the middle point of the minor or major arc connecting $m_1$ and $m_2$. 

For $m_1\ne m_2$, 
in addition to the equilateral triangle already mentioned,
there is one isosceles triangle
with $\cos(\theta_2-\theta_1)=(\sqrt{2}-1)/2$, where the mass $m_3$ is located at the middle point of the minor arc connecting $m_1$ and $m_2$,
that is 
 $\theta_2-\theta_3=\theta_3-\theta_1=(\theta_2-\theta_1)/2$.
This rigid rotator (the shape) is also masses independent.

For the cotangent potential \eqref{partpot} we found an example where the minimum number of rigid rotators on rotating meridian
is two and the maximum number is six.
We expect that this should be true in general
for cotangent potential.
However, it is still an open question.
It needs further investigations.

Finally, we must emphasises the importance of the condition $c_x=c_y=0$ in (\ref{CxCy3}).
Comparing with the Newtonian $n$-body problem on the
Euclidean plane, we have lost two important first integrals corresponding to
the centre of mass.
Fixing the centre of mass at the origin
we get clear representations for relative equilibria. 
For a given shape, given mutual distances $r_{ij}$, the position of the centre of mass is clear. For $RE$ on $\mathbb{S}^2$, 
we lost the above two integrals,
but we get two first integrals $c_x$ and $c_y$ instead.
As we have shown, these
two integrals determine the rotation axis $z$.
Indeed,
for the limit $R\to \infty$ with $r_k=R\theta_k$ finite,
$c_x=c_y=0$ takes the form
\begin{equation}
\sum_k m_k r_k \cos\phi_k +O(1/R^2)
= \sum_k m_k r_k \sin\phi_k +O(1/R^2)
=0.
\end{equation}
Which is just the centre of mass condition
for the polar coordinate of the Euclidean plane.
Thus, $c_x=c_y=0$ is the extension of the center of mass condition
on the plane to rotation axis condition on $\mathbb{S}^2$.

From the condition $c_x=c_y=0$,
we find the translation formula 
between the shape variables $\theta_i-\theta_j$
and the configuration variables $\theta_k$,
given by the equations (\ref{sin2ThetaByDeltaTheta1}--\ref{cos2ThetaByDeltaTheta1}).
They are the key equations for the condition
of the rigid rotator on a meridian.
Since the conditions $c_x=c_y=0$ are just 
based on the $SO(3)$ invariance and the definition of the
relative equilibria,
the formula (\ref{sin2ThetaByDeltaTheta1}--\ref{cos2ThetaByDeltaTheta1})
and the method to find the rigid rotator and relative equilibria
 can be applicable to a wide class of potentials,
including attractive forces, repulsive forces,
and for charged particles.

\appendix
\section{Direct derivation of the equation (\ref{theCosineTheorem})
from (\ref{theSineTheorem})} \label{directCosineTheorem}

Let $a=\phi_1-\phi_2$ and $b=\phi_2-\phi_3$.
Then, $\sin(\phi_3-\phi_1)=-\sin(a+b)$.
By the equation (\ref{theSineTheorem}),
\begin{equation*}
\begin{split}
-\rho\mu_2
&=\sin(a+b)\\
&=\sin(a)\cos(b)+\cos(a)\sin(b)\\
&=\rho\mu_3\cos(b)+\rho\mu_1\cos(a).
\end{split}
\end{equation*}
Therefore,
\begin{equation*}
\begin{split}
\mu_3^2\cos^2(b)
&=\mu_2^2+2\mu_1\mu_2\cos(a)+\mu_1^2\cos^2(a).\\
\Leftrightarrow
\mu_3^2(1-\rho^2\mu_1^2)
&=\mu_2^2+2\mu_1\mu_2\cos(a)
	+\mu_1^2(1-\rho^2\mu_3^2).\\
\Leftrightarrow
\mu_3^2
&=\mu_2^2+2\mu_1\mu_2\cos(a)+\mu_1^2.
\end{split}
\end{equation*}
Finally, we obtain
\begin{equation*}
\cos(\phi_1-\phi_2)
=\frac{\mu_3^2-(\mu_1^2+\mu_2^2)}{2\mu_1\mu_2}.
\end{equation*}
Similarly, we obtain the equations in (\ref{theSineTheorem}).

\section{Exceptional cases for rigid rotator
on a rotating meridian} \label{exceptionalCases}

\subsection*{The case 2.
$G_{12}=G_{23}$ and $F_{12}=F_{23}$}
In this subsection,
the exceptional case 2,
$G_{12}=G_{23}$ and $F_{12}=F_{23}$
surely occur for special angles
$\theta_1-\theta_2$ and $\theta_2-\theta_3$
for given $\nu_1=m_1/m_3$.

The condition $G_{12}=G_{23}$ and $F_{12}=F_{23}$
are
\begin{equation}
\begin{split}
m_1m_2\sin(2(\theta_1-\theta_2))&=m_2m_3\sin(2(\theta_2-\theta_3)),\\
\frac{m_2m_3\sin(\theta_1-\theta_2)}{|\sin(\theta_1-\theta_2|^3}&=\frac{m_2m_3\sin(\theta_2-\theta_3)}
{|\sin(\theta_2-\theta_3)|^3}.
\end{split}
\end{equation}

In the following, let us write
$\theta_{12}=\theta_1-\theta_2$,
$\theta_{23}=\theta_2-\theta_3$,
and $\nu_1=m_1/m_3$.
We can assume $\sin\theta_{12}>0$ without loss of generality.
Then, by the second equation $\sin\theta_{23}>0$,
the equations are
\begin{equation}
\begin{split}
\nu_1\sin\theta_{12}\cos\theta_{12}&=\sin\theta_{23}\cos\theta_{23},\\
\nu_1^{-1/2}\sin\theta_{12}&=\sin\theta_{23}.
\end{split}
\end{equation}
Divide the first line by the second line to get
\begin{equation}
\nu_1^{3/2}\cos\theta_{12}=\cos\theta_{23}.
\end{equation}

The solutions are
\begin{equation}
\label{exceptionalTheta12}
\sin\theta_{12}
=\sqrt{\frac{\nu_1(1+\nu_1+\nu_1^2)}{(1+\nu_1)(1+\nu_1^2)}
	}\,\,,
\cos\theta_{12}
=\pm\frac{1}{\sqrt{(1+\nu_1)(1+\nu_1^2)}}\,\,,
\end{equation}
\begin{equation}
\label{exceptionalTheta23}
\sin\theta_{23}
=\sqrt{\frac{1+\nu_1+\nu_1^2}{(1+\nu_1)(1+\nu_1^2)}
	}\,\,,
\cos\theta_{23}
=\pm\sqrt{\frac{\nu_1^3}{(1+\nu_1)(1+\nu_1^2)}}\,\,.
\end{equation}

Note that the absolute value of the right hand side of each equation is always smaller than $1$ for any mass ratio $\nu_1=m_1/m_3$.
Therefore, the solutions  exist for any $\nu_1$.

\subsection*{The case 3.
$G_{12}=G_{31}$ and $F_{12}=F_{31}$}
Replace
$\theta_{23}\to \theta_{31}=\theta_3-\theta_1$ and 
$\nu_1\to \nu_2=m_2/m_3$ in the 
equations (\ref{exceptionalTheta12}) and (\ref{exceptionalTheta23})
to get the relations between $\nu_2=m_2/m_3$
and $\theta_{12}$, $\theta_{31}$.
Namely,
\begin{equation}
\label{exceptionalTheta12ForCase3}
\sin\theta_{12}
=\sqrt{\frac{\nu_2(1+\nu_2+\nu_2^2)}{(1+\nu_2)(1+\nu_2^2)}
	}\,\,,
\cos\theta_{12}
=\pm\frac{1}{\sqrt{(1+\nu_2)(1+\nu_2^2)}}\,\,,
\end{equation}
\begin{equation}
\label{exceptionalTheta31}
\sin\theta_{31}
=\sqrt{\frac{1+\nu_2+\nu_2^2}{(1+\nu_2)(1+\nu_2^2)}
	}\,\,,
\cos\theta_{31}
=\pm\sqrt{\frac{\nu_2^3}{(1+\nu_2)(1+\nu_2^2)}}\,\,.
\end{equation}

\subsection*{The case 4. $G_{12}=G_{23}=G_{31}$
and $F_{12}=F_{23}=F_{31}$}
In this subsection, it is shown that
the case 4, $G_{12}=G_{23}=G_{31}$ and
$F_{12}=F_{23}=F_{31}$ is satisfied
only when $m_1=m_2=m_3$
and $\theta_1-\theta_2=\theta_2-\theta_3=\theta_3-\theta_1=2\pi/3$.

The condition of $F_{12}=F_{23}=F_{31}$ is already solved
in the section \ref{sec4}.
The solution is
\begin{equation}
\label{cosAndSinForException1}
\begin{split}
\cos(\theta_i-\theta_j)&=\frac{\mu_k^2-(\mu_i^2+\mu_j^2)}{2\mu_i\mu_j},\\
\sin(\theta_i-\theta_j)&=\rho \mu_k.
\end{split}
\end{equation}
The definition of $\mu_k$ is the same as in the section \ref{sec4} and $\rho$ is given by (\ref{solForK}).

Substituting this solution into the condition for $G_{ij}$,
we obtain
\begin{equation}
a_k=\mu_k^4\Big(\mu_k^2-(\mu_i^2+\mu_j^2)\Big)
\mbox{ are common}.
\end{equation}
Then,
\begin{equation}
\begin{split}
0&=a_1-a_2
=(\mu_1^2-\mu_2^2)
	\Big(\mu_1^4+\mu_2^4-\mu_3^2(\mu_1^2+\mu_2^2)\Big),\\
0&=a_2-a_3
=(\mu_2^2-\mu_3^2)
	\Big(\mu_2^4+\mu_3^4-\mu_1^2(\mu_2^2+\mu_3^2)\Big).
\end{split}
\end{equation}

From the first equation, we obtain
$\mu_1=\mu_2$ or 
$\mu_1^4+\mu_2^4-\mu_3^2(\mu_1^2+\mu_2^2)=0$.
If $\mu_1=\mu_2$, putting this into the second equation
yields $\mu_1=\mu_2=\mu_3$.

Similarly, if $\mu_2=\mu_3$,
the second equation is satisfied,
and the first equation yields $\mu_1=\mu_2=\mu_3$.

So, the remaining possibility to be considered are the case
\begin{equation}
\begin{split}
\mu_1^4+\mu_2^4-\mu_3^2(\mu_1^2+\mu_2^2)&=0,\\
\mu_2^4+\mu_3^4-\mu_1^2(\mu_2^2+\mu_3^2)&=0.
\end{split}
\end{equation}
Subtracting the second equation from the first,
we obtain
$(\mu_1^2-\mu_3^2)(\mu_1^2+\mu_2^2+\mu_3^2)=0$.
Therefore, $\mu_1=\mu_3$.
Substituting $\mu_1=\mu_3$ into the first equation,
we again obtain $\mu_1=\mu_2=\mu_3$.

Then, by the equation (\ref{cosAndSinForException1}),
$\cos(\theta_i-\theta_j)=-1/2$.
Namely,
$\theta_1-\theta_2
=\theta_2-\theta_3
=\theta_3-\theta_1
=2\pi/3$.

\subsection*{Acknowledgements}
The second author (EPC) has been partially supported 
by Asociaci\'on Mexicana de Cultura A.C. and Conacyt-M\'exico Project A1S10112.


\begin{thebibliography}{99}

\bibitem{JA} Andrade J., D\'avila N., Perez-Chavela E. and Vidal C.
\emph{Dynamics and regularization of the Kepler problem on surfaces of constant curvature}, Can. J. Math.,  {\bf 69}-5, (2017), 961-991. 

\bibitem{Bengochea} Bengochea A., Garc\'ia-Azpeitia C., P\'erez-Chavela E.,  Roldan P. {\it Continuation of relative equilibria in the $n$--body problem to spaces of constant curvature}  
Journal of Differential Equations, \textbf{307}, (2022), 137-159.

\bibitem{Borisov1} Borisov A.V., Mamaev I.S., Bizyaev I.A. \emph{The Spatial Problem of $2$ Bodies on a Sphere, Reduction and Stochasticity }, Regular and Chaotic Dynamics {\bf 216}-5, (2016), 556-580.

\bibitem{Borisov2} Borisov A. V.,  Mamaev I. S., Kilin  A. A., {\it Two-body problem on a sphere: reduction, stochasticity, periodic orbits}; Institute of Computer Science, Udmurt State University, (2005).

\bibitem{Borisov3} Borisov A. V., Mamaev I. S., {\it The restricted two-body problem in constant curvature
spaces}, Celestial Mech. Dynam. Astronom \textbf{96}, 1-17, (2006).

\bibitem{Carinena} Cari\~{n}ena J. F., Ra\~{n}ada M. F. y Santander M. {\it Central potentials on spaces of constant curvature: the Kepler problem on the two-dimensional sphere 
$S^2$ 
and the hyperbolic plane $H^2$.} J. Math. Phys. \textbf{46} (2005), no. 5, 052702, 25 pp.

\bibitem{Diacu-EPC1} Diacu F., P\'erez-Chavela E., Santoprete M., {\it The n-body problem in spaces of constant curvature. Part I: Relative equilibria.} J. Nonlinear Sci. \textbf{22} (2012), no. 2, 247--266.

\bibitem{Diacu-EPC2} Diacu F., P\'erez-Chavela E., Santoprete M.,, {\it The n-body problem in spaces of constant curvature. Part II: Singularities.} J. Nonlinear Sci. \textbf{22} (2012), no. 2, 267--275.

\bibitem{Diacu1} Diacu F., Relative equilibria of the curved N-body problem. Atlantis Studies in Dynamical Systems, Atlantis Press, Amsterdan, Paris, Beijing \textbf{1}, 2012.

\bibitem{Diacu2} Diacu F.,
\emph{Polygonal Homographic  Orbits of the Curved $n-$Body Problem}, Trans. Amer. Math. Soc.
 {\bf 364}-5, (2012), 2783-2802.

\bibitem{Diacu3} Diacu F.and P\'erez-Chavela E.,
\emph{Homographic solutions of the curved 3-body problem}, Journal of Differential Equations {\bf 250},  (2011), 340-366.

\bibitem{Diacu4} Diacu F., S\'anchez-Cerritos J.M. and Zhu S.
\emph{Stability of Fixed Points and Associated Relative Equilibria of the 3-body Problem on $S^1$ and $S^2$}, Journal of Dynamics and Differential Equations, {\bf 30}, (2018), 209-225.

\bibitem{Euler} Euler L. De mutuo rectilineo trium corporum se mutuo attrahentium, Novi Comm. Acad. Sci. Imp. Petrop. 11 (1767) 144-151.

\bibitem{Hestenes}David Hestenes,
{\it New foundation for classical mechanics},
Kluwer Academic Publishers,
Second edition, 2003

\bibitem{Kozlov} Kozlov V. V.,  Harin A. O.,  {\it Kepler's problem in constant curvature spaces.} Cel. Mech. Dynam. Astronom. \textbf{54} (1992), no. 4, 393--399.
 
\bibitem{M-S} Mart\'inez R. Sim\'o C.,  {\it On the stability of the Lagrangian homographic solutions in a  curved three body problem on $\mathbb{S}^2$.} Discrete Cont. Dyn. Syst. Ser. A. {\bf 33} (2013), 1157--1175. 

\bibitem{Moeckel}
Moeckel R., Notes on Celestial Mechanics (especially central
configurations). http://www.math.umn.edu/\~rmoeckel/notes/Notes.html

\bibitem{EPC1} P\'erez-Chavela E. and Reyes-Victoria J.G.,
 \emph{An intrinsec approach in the curved $n$-body problem. The positive curvature case}, Trans. Amer. Math. Soc. {\bf 364}-7, (2012),  3805-3827.
 
 \bibitem{EPC2} P\'erez-Chavela E. and S\'anchez-Cerritos J.M. \emph{Euler-type relative equilibria in spaces of constant curvature and their stability}, Canad. J.  Math. {\bf 70}-2, (2018), 426-450.

\bibitem{Shchepetilov2} Shchepetilov A.V., {\it Nonintegrability of the two-body problem in constant curvature spaces.}  J. Phys. A \textbf{39} (2006), no. 20, 5787--5806. 

\bibitem{tibboel} Tibboel P.,  {\it Polygonal homographic orbits in spaces of constant curvature.} Proc. Amer. Math. Soc.  \textbf{141} (2013), 1465-1471.

\bibitem{Vozmischeva2} Vozmischeva T. G., {\it Integrable problems of celestial mechanics in spaces of constant curvature}, Kluwer Acad. Publ., Dordrecht, (2003).

\bibitem{Wintner} Wintner A., {\it The
Analytical Foundations Celestial of Mechanics}, Princeton
University Press, Princeton, New York, (1941).

\bibitem{zhu} Zhu S., {\it Eulerian relative equilibria of the curved 3-body problem.} Proc. Amer. Math. Soc. \textbf{142} (2014), 2837-2848.

\end{thebibliography}
\end{document}